\documentclass[11pt]{article}
\usepackage{graphicx, tikz}

\usepackage{amssymb, amsmath}
\usepackage{theorem}
\usepackage[colorlinks]{hyperref}

\numberwithin{equation}{section}

\newtheorem{thm}{Theorem}[section]
\newtheorem{lemma}[thm]{Lemma}
\newtheorem{sublemma}[thm]{Sublemma}

\newtheorem{prop}[thm]{Proposition}
\newtheorem{cor}[thm]{Corollary}
{\theorembodyfont{\rmfamily}

\newtheorem{rmk}[thm]{Remark}
}

\newcommand{\qed}{\hfill \mbox{\raggedright \rule{.07in}{.1in}}}

\newenvironment{proof}{\vspace{1ex}\noindent{\bf
Proof}\hspace{0.5em}}{\hfill\qed\vspace{1ex}}
\newenvironment{pfof}[1]{\vspace{1ex}\noindent{\bf Proof of
#1}\hspace{0.5em}}{\hfill\qed\vspace{1ex}}

\newcommand{\R}{{\mathbb R}}

\newcommand{\Z}{{\mathbb Z}}
\newcommand{\D}{{\mathbb D}}
\newcommand{\N}{{\mathbb N}}
\renewcommand{\P}{{\mathbb P}}

\newcommand{\B}{{\mathcal B}}

\newcommand{\E}{\mathbb{E}}

\newcommand{\eps}{\varepsilon}

\newcommand{\ie}{{\em i.e.,} }

\usepackage{mathtools}
\DeclarePairedDelimiter\floor{\lfloor}{\rfloor}

\title{Limit theorems for wobbly interval intermittent maps}

\author{Douglas Coates
\thanks{Department of Mathematics, University of Exeter,
North Park Road, Exeter EX4 4QF, UK.
Email address: {\it dc485@exeter.ac.uk}}
\and Mark Holland
\thanks{Department of Mathematics, University of Exeter,
North Park Road, Exeter EX4 4QF, UK. Email address:
{\it M.P.Holland@exeter.ac.uk}}
\and Dalia Terhesiu 
\thanks{Department of Mathematics, University of Leiden, Niels Bohrweg 1,
2333 CA Leiden, The Netherlands. Email address: {\it daliaterhesiu@gmail.com}}
}

\begin{document}
\baselineskip=17pt
\maketitle

\abstract{We consider perturbations of interval maps with indifferent fixed points, which we refer to as wobbly  interval intermittent maps, for which stable laws for general H{\"o}lder
observables fail. We obtain limit laws for such maps and   H{\"o}lder
observables. These limit laws are similar to the classical semistable laws previously established for random processes.
One of the considered examples is an interval map with a countable number of 
discontinuities, and to analyse it we need to construct a Markov/Young tower. }

\section{Introduction and main results}

\subsection{Formulating the problem and statement of the main result}
\label{sec-infst}
In this paper we study the map
\begin{equation}\label{eq:nMPM}
 f_M(x) = \begin{cases}
        x(1 + M(x)x^\alpha),& \text{ if } x \in [0,\frac12];\\
        2x-1, & \text{ if } x \in (1/2,1].
        \end{cases}
\end{equation}
with $\alpha \in (0,1)$, for oscillatory functions $M$ to be described further
bellow. We start by recalling that when $M(x)\equiv 2^{\alpha}$, $f_M$ is the well known
Liverani-Saussol-Vaienti (LSV) map~\cite{LSV99}, which preserves an absolutely continuous w.r.t.\ Lebesgue invariant measure $\mu$.
When $\alpha\in (0,1)$ the measure $\mu$ can be rescaled to give an invariant probability measure, while for $\alpha\ge 1$,
the measure $\mu$ is infinite, $\sigma$ finite. In last two decades, several statistical properties for the LSV map have been established
in both the probability setting, including~\cite{LSV99, Young, Sarig02, Gouezel04, Gouezel04b}
and the infinite measure setting, including~\cite{ThalerZweimuller06, MT12, Gouezel11}. This list of references is by far not exhaustive and
several important references can be found in the above mentioned works.

 As relevant for the present work, we recall that in the probability case of~\eqref{eq:nMPM} with $M(x)\equiv 2^{\alpha}$, $\alpha\in (0,1)$, the following limit theorems have been established:
\begin{itemize}
\item[a)] If $\alpha\in (0,1/2)$, then the classical central limit theorem holds for H{\"o}lder observables on $[0,1]$ 
(see \cite{LSV99} and Young~\cite{Young}). If  $\alpha=1/2$
and $v$ is H{\"o}lder then: i) if  $v(0)=0$, the classical central limit theorem holds; ii)  if $v(0)\ne 0$,  then $v\notin L^2$ and we have convergence to the Gaussian $\mathcal{N}(0,1)$ but with
non standard normalisation. We refer to Gou{\"e}zel~\cite[Section 1.1.2]{ Gouezel04b}
for details.
\item[ b)] If $\alpha\in (1/2,1)$, the following two possibilities occur for H{\"o}lder observables $v: [0,1]\to\R$  with $\int v\, d\mu=0$ (see~\cite{ Gouezel04b}):
\begin{itemize}
\item[i)] If $v$ is $\nu$-H{\"o}lder with $\nu>\alpha-1/2$ and $v(0)=0$, then the classical central limit theorem holds for such $v$.
\item[ii)]  If $v(0)\ne 0$, then  the Birkhoff sum $S_n^{f_M} v(x)=\sum_{k=0}^{n-1}v(f_M^j(x))$ converges in distribution to a stable  law of index $\beta=1/\alpha$. For background on stable laws we refer to
~\cite{Feller} (see also~\cite[Section 1.1.2]{ Gouezel04b} for a brief summary).
\end{itemize}
\end{itemize}

In this work we study the type of limit laws that occur when the function  
$M:[0,1]\to\R$ appearing in \eqref{eq:nMPM} with $\alpha\in (0,1)$ is an oscillatory function 
and refer to  these maps as \emph{wobbly intermittent map}s. As clarified in Subsection~\ref{sec-mainres}, when $\alpha\in (1/2, 1)$, for the choices of $M=M_1, M=M_2$ introduced below 
(in~\eqref{eq.intermittent2} and~\eqref{eq:def-M-sine}), 
general  H{\"o}lder observables do not give rise to stable laws.  The main purpose of this paper is to understand what type
of limit theorems can be enjoyed by  H{\"o}lder observables in the setup of wobbly intermittent maps $f_{M_1}, f_{M_2}$  introduced below.

We focus on the wobbly map $f_{M_1}:[0,1]\to[0,1]$ defined  as in \eqref{eq:nMPM}  with $M=M_1: [0,1]\to (0,\infty)$ given by
\begin{equation}\label{eq.intermittent2}
M_1(x)=\frac{C_0}{2^{\{c^{-1}_1\log x\}}}, \text{ where } c_1>0\text{ and }C_0\text{ so that }
f_{M_1}\left(\left(1/2\right)^{-}\right)=1.
\end{equation}
Here and throughout $\{\cdot\}$ denotes the fractional part. 
When $c^{-1}_1\log 2<1$, $C_0=2^{\alpha-\frac{1}{c_1}\log 2}$.
This choice 
corresponds to $M_1$ satisfying a \emph{log-periodic} condition with period $e^{c_1}$, namely
$M_1(e^{c_1}x)=M_1(x)$ for all $x>0$.  Here and throughout we define $\{x\}=\{-x\}$ for $x<0$.

As shown in  Section~\ref{sec.first-tail}, for $M_1$ as in~\eqref{eq.intermittent2} and $\alpha\in (0,1)$,
the tail of the first return time of $f_{M_1}$ to $[1/2,1]$ is somewhat similar to the generalised St. Petersburg distribution with index $\beta=1/\alpha\in (1,2)$.
We recall that \emph{semistable} laws for i.i.d random variables $X, X_1, X_2\ldots$ with a common generalised St. Petersburg distribution $\P(X>y)= 2^{\{\beta\log_2 y\}} y^{-\beta}$ for $y\ge 2^{1/\beta}$ with  $\beta\in (0,2)$   have been obtained
by Cs{\"o}rg{\H{o}}~\cite{Csorgo} (see also~\cite{CM1, KC, PK1}).  For background on semistable laws we refer to Megyesi~\cite{M},  Cs{\"o}rg{\H{o}} and  Megyesi~\cite{CM1} as well as 
~\cite{Csorgo, KC, PK1} and references therein (see also Section~\ref{sec-sst} below for a brief introduction). 
It is known the induced map of the Gaspard-Wang map~\cite{GW88},
a linearised version of the map~\eqref{eq:nMPM}, is isomorphic to an i.i.d.\  process.  Limit laws of  \emph{infinite} measure preserving \emph{Gaspard Wang} maps with
a distribution of return times behaving similar to a \emph{St.\ Petersburg distribution} (of index $\beta\in (0,1)$)  have been obtained by Kevei and Terhesiu~\cite{KT18}
(in particular, see~\cite[Section 5]{KT18}). The present  Theorem~\ref{thm-infst}  for finite measure preserving wobbly maps is partially motivated by the results in~\cite{KT18}. We believe that a \emph{modified Theorem~\ref{thm-infst}} 
for \emph{infinite} measure preserving wobbly maps (so, $f_{M_1}$ as in \eqref{eq:nMPM}  with $M_1$ as in~\eqref{eq.intermittent2} and $\alpha\ge 1$)
can be obtained by a straightforward combination of the results in present paper with the ones in~\cite{KT18}, but here we {\emph{focus on the finite} measure case.

A simplified wobbly map, which as explained below induces with respect to first return to a Gibbs Markov map,  is $f_{M_2}:[0,1]\to[0,1]$ defined  as in \eqref{eq:nMPM}  with a differentiable  $M=M_2$ given by
\begin{equation}\label{eq:def-M-sine}
    M_2(x)= 
a\left(1+ b\sin\left(\frac{2\pi}{c_2}\log(x)\right)\right)
\end{equation}
where $a,c_2$ are positive real constants with $b<1/2$ so that $M_2$ is bounded and bounded away from zero
and $M_2$ is log periodic with period $e^{c_2}$, that is $M_2(e^{c_2}x)=M_2(x)$ for all $x>0$.


Throughout we let $\to^d$ stand for convergence 
in distribution with respect to the any measure that is absolutely continuous with respect to the $f_{M_1}, f_{M_2}$-invariant measure $\mu$. With these specified we provide the statement of the main result.
For the statement below we let $V_{1/\alpha}(c)$ be a random variable in the domain of a semistable law of index $1/\alpha\in (1,2)$ and parameter
$c>1$; for a precise definition of such random variable we refer to Section~\ref{sec-sst}. Also, throughout  $ \lfloor \, \rfloor$ stands for the integer part.

\begin{thm}
\label{thm-infst}Let $M_1$ and $M_2$ as in~\eqref{eq.intermittent2} and ~\eqref{eq:def-M-sine}.  Assume the setup of either $f_{M_1}$ or $f_{M_2}$ with 
 $\alpha\in (0,1)$.

 Let $v: [0,1]\to\R$ be  a H{\"o}lder observable. 
If either $\alpha\in (0,1/2)$ or  $\alpha\in [1/2, 1)$ and $v(0)=0$, then  the classical central limit theorem holds for $f_{M_1}$ and $f_{M_2}$. Also,
the following non-Gaussian limit theorem holds:
\begin{itemize}
\item[i)] Consider the setup of $f_{M_1}$ with $\alpha\in (1/2, 1)$.  
Suppose that $v$ is H{\"o}lder of exponent $\nu>\alpha-1/2$  so that $v(0)\ne 0$. 
Set $c=e^{\alpha c_1}$ and $k_n = \lfloor c^n \rfloor$. 
Then there exists a subsequence $n_r$  of $k_n$ so that as $r\to\infty$
\[
\frac{\sum_{j=0}^{n_r-1} v\circ f_{M_1}^{j}-B_{n_r}\int v\, d\mu}{A_{n_r}}\rightarrow^d W
\]
for some random variable $W$ with a non trivial distribution, where $A_{n_r}=(n_r)^{\alpha}$ and $B_{n_r}$ is so that $\lim_{r\to\infty}B_{n_r}/n_r$ exists.

\item[ii)] Consider the setup of $f_{M_2}$ and the same assumptions on $\alpha$ and $v$ as  in item i). Set $c=e^{\alpha c_2}$, $k_n = \lfloor c^n \rfloor$ and  $A_{k_n}=(k_n)^{\alpha}$.
 In this case we have the improved result
\begin{equation*}
\frac{\sum_{j=0}^{k_n-1}v\circ f_{M_2}^{j}-B_{k_n}\int v\, d\mu}{A_{k_n}}\rightarrow^d V_{1/\alpha}(c),\text{ as } n\to\infty,
\end{equation*} 
where $B_{k_n}$ is so that $\lim_{n\to\infty}B_{k_n}/k_n$ exists.
\end{itemize}
\end{thm} 
Given Theorem~\ref{thm.singmap} and Remark~\ref{rmk-sin} below, we note that  the ranges of $\alpha$ and assumptions on $v$ which yield the central limit theorem for these wobbly maps,
are essentially the same as in standard case~\eqref{eq:nMPM}, (i.e. with 
$M(x)\equiv 2^{\alpha}$) and the corresponding proof goes through in exactly the same
way (as in the works acknowledged in a)--b) above).
The proof of Theorem~\ref{thm-infst} for ranges of $\alpha$  and conditions on $v$ that do \emph{not}  give rise to the central limit theorem is deferred to Section~\ref{sec-exactst},  where we also clarify that this holds for larger classes of  subsequences.

\begin{rmk}
\label{rmk-diffsemcl} It seems likely that Theorem~\ref{thm-infst} i) which gives the limit behaviour of the Birkoff sums for $f_{M_1}$ can be improved to the form of Theorem~\ref{thm-infst} ii).
But this requires a refinement of Theorem~\ref{thm.singmap} below to $\mu(\varphi>n)=C\mu(\tau>n)(1+o(1))$ for some $C>0$; this type of result would require a serious amount of new work,
which we cannot pursue here.
\end{rmk}

\begin{rmk}
We believe that the proof of Theorem~\ref{thm-infst} can be extended to treat the case 
$\alpha=1/2$, when  $v$ is H{\"o}lder with $v(0)\ne 0$;  we expect convergence
along subsequences to the Gaussian $\mathcal{N}(0,1)$ with non standard normalisation. This seems plausible given the recent
work of Berkes~\cite{Berkes}, but we do not treat this case here. 
\end{rmk}
\begin{rmk}
A further extension of Theorem~\ref{thm-infst} is to treat
the case $\alpha=1$, but with a suitably modified $f$, such as
\begin{equation}
\label{eq-holland}
f_{M}^1(x) = \begin{cases}
        x (1+ M(x)x(\log x)^2),& \text{ if } x \in [0,\frac12];\\
        2x-1, & \text{ if } x \in (1/2,1].
        \end{cases}
\end{equation}
with $M\equiv M_1$ as in~\eqref{eq.intermittent2} or $M\equiv M_2$ as in~\eqref{eq:def-M-sine}.
Special versions of these maps $f_{M}^1$ with $M\equiv 1$ have been studied in 
see Holland~\cite{Holland05}.
The extra factor $(\log x)^2$ ensures that  $f_{M_1}^1$ preserves a finite measure $\mu\ll Leb$. 
In particular, for $M\equiv 1$,~\cite{Holland05} shows that the tail of the first return time of $f_{M_1}$ to $[1/2,1]$ behaves asymptotically as $n^{-1}(\log n)^{-2}$. We expect a semi-stable
law to hold for this example, but the proof requires a more elaborate asymptotic analysis
of the tail $\mu(\tau>y)$, which we do not present here.
\end{rmk}

\subsection{Method of proof and main dynamical ingredient}
\label{sec-mainres}

A standard way of obtaining limit theorems, in particular stable laws, for maps with indifferent fixed points is via inducing on sets away from the indifferent fixed points
along with understanding the return time distribution. To outline the difficulties in the setup of the wobbly maps $f_{M_1}$, and $ f_{M_2}$ 
(with $M_1, M_2$ as in~\eqref{eq.intermittent2} and ~\eqref{eq:def-M-sine}),
we briefly recall the main steps of the involved analysis in the probability preserving LSV map $([0,1],\mathcal{B}([0,1]), f, \mu)$ as in~\eqref{eq:nMPM} with $M\equiv 2^\alpha$. Let $Y\subset (0,1]$ be a reference set and let $\tau:Y\to\mathbb{N}$,
$\tau(x) = \inf \{ n \geq 1 : f^n(x) \in Y\}$ be the first return time of $f$ to $Y$. The main steps are:
\begin{itemize}
\item[a)] Recall that the induced map $f^{\tau}:Y\to Y$, which   preserves the measure $\mu_Y(E):=\mu(Y\cap E)$,
has good functional analytic properties in a Banach space  $\B\subset L^\infty(Y)$ (for instance, the space of piecewise H{\"o}lder continuous functions); 
\item[b)] Recall that $\mu_Y\{\tau>n\}=Cn^{-1/\alpha}(1+o(1))$  for some $C>0$ depending only on $f$. See~\cite{LSV99, Young} and also~\cite{Holland05, T15, T16}  for improved asymptotics.
\item[c)]  Obtain a limit theorem (stable law) via the spectral method, the Nagaev-Guivarc'h (also referred to as Aaronson-Denker) method, for $f^\tau$; we refer to the comprehensive survey~\cite{Gou15}
and to the original papers~\cite{AaronsonDenker01, GH88}.
\item[ d)] Pull back the limit theorem from the induced system via the method of Melbourne and T{\"o}r{\"o}k~\cite{MTor} (see also Sarig~\cite{Sarig01}, Zweim{\"u}ller~\cite{Zweimuller07}
for details of the method for the stable law case) or the method in~\cite{ Gouezel04b} based on techniques related to operator renewal sequences. While the `pull back' method is more elementary and somewhat less technical, the second has the advantage of  
eventually providing error rates in the involved convergence (such as Berry-Essen) as well as  local limit theorems as in Gou{\"e}zel~\cite{ Gouezel05}.
\end{itemize}
In the sequel we adapt the steps a)--d) to the proof of 
Theorem~\ref{thm-infst}. A first, more or less obvious, difference is that in the setup of the wobbly maps $f_{M_1}$ and $f_{M_2}$
\emph{item b) does not hold};  for the asymptotics of the return time in the setup of $f_{M_1}, f_{M_2}$ and $f_M^1$ defined in~\eqref{eq-holland} we refer to equation~\eqref{eq:tail-of-tau-sin-final} in Section~\ref{sec.first-tail}.
We recall that stable laws of index $\beta\in (0,2)$ for Gibbs Markov maps hold if and only if $\mu_Y\{\tau>n\}$ is regularly varying with index $-\beta$, that is
if and only if
$\mu_Y\{\tau>n\}=Cn^{-\beta}\ell(n)$  for $\beta\in (0,2)$ and $\ell$ some slowly varying\footnote{A measurable function $\ell: (0,\infty)\to (0,\infty)$ is 
slowly varying if $\lim_{x\to\infty}\frac{\ell(\lambda x)}{\ell(x)}\to 1$ for each $\lambda > 0$.} function; we refer to~\cite[Theorem 1.5]{Gouezel10b}. This together
with equation~\eqref{eq:tail-of-tau-sin-final} in Section~\ref{sec.first-tail} explains why stable laws cannot hold for $f_{M_1}$ and $ f_{M_2}$.

Items c)-d) are dealt with in an abstract setup where regular variation ( in particular, item b) ) fails. We refer to assumption (H3) in  Section~\ref{sec:abstra} and to the abstract Theorem~\ref{prop-main}.
In Section~\ref{sec-exactst} (summarising the various technical results in  Sections~\ref{sec.first-tail} and~\ref{proof.thm.sing}) we verify that the abstract assumptions of Theorem~\ref{prop-main}
hold in the setup of $f_{M_1}$, and $f_{M_2}$.

It is known that the induced map of the LSV map~\eqref{eq:nMPM} is Gibbs Markov, which facilitates an easy verification of item a) above. Roughly, a Gibbs Markov map is a uniformly expanding map with big images and good distortion properties;
see~\cite[Chapter 4]{Aaronson} (and Section~\ref{sec:abstra} below) for a complete definition.
While the first return map of the  wobbly map  $f_{M_2}$ is Gibbs Markov, the situation is very different in the setup of the wobbly map $f_{M_1}$.
The first difficulty in analysing the map $f_{M_1}$~\eqref{eq.intermittent2} is that there is a singularity set $\mathcal{C}=\{s_{\ell}=e^{-\ell c_1},\ell\geq 1\}$, namely the set of points $x\in(0,1/2)$
for which $-c^{-1}_1\log x\in\mathbb{N}$. We'll assume that $c$ is chosen so that $s_1=e^{-c_1}\ll 1$.

We write $\mathcal{D}=e^{-c_1}$, and $I_{\mathcal{D}}=[0,\mathcal{D}]$. 
Hence $\mathcal{C}\subset I_{\mathcal{D}}$. Thus we have to apply re-inducing schemes, and show the existence of
a return time function $\varphi:Y\to Y$ for which $(f_{M_1})^\varphi$ is Gibbs Markov. Furthermore we have to show that
the asymptotics of the tails of $\varphi$ are comparable to those of the first return time $\tau$.

We build upon the combinatorial arguments of Bruin, Luzzatto and van Strien~\cite{BLvS} to establish the following result, 
of which the proof is deferred to Section \ref{proof.thm.sing}.

\begin{thm}\label{thm.singmap}
Let $Y=[1/2,1]$. Then there exists a countable partition $\mathcal{Q} \pmod \mu$ of $Y$ into subintervals $\{Y_i\}$ and a return time function
$\varphi:\mathcal{Q}\to\mathbb{N}$ such that $F:Y\to Y$ where $F=f_{M_1}^{\varphi}$ or  $F=(f_{M_1}^1)^{\varphi}$ is a Gibbs-Markov map with respect to $\mathcal{Q}$, 
preserving an absolutely continuous (w.r.t.\ Lebesgue) probability measure $\mu$  so that for some $C>0$,
\begin{equation}\label{eq.return}
\mu(\varphi>n)\le C\,n^{-\frac{1}{\alpha}}.
\end{equation}
\end{thm}
Theorem~\ref{thm.singmap} can be be regarded as a substantial refinement of the work of Diaz-Ordaz,
Holland and Luzzatto~\cite{DHL06} which adapts arguments in~\cite{BLvS}
to reinduce maps with a\emph{ finite number of singularities/discontinuities} to Gibbs Markov maps.

\begin{rmk}
\label{rmk-sin}
A similar Theorem~\ref{thm.singmap} holds for the tails of the \emph{first} return time $\tau$ of $f_{M_2}$ to  $Y=[1/2,1]$; this follows from ~\eqref{eq:tail-of-tau-sin-final} in Section~\ref{sec.first-tail}.
The sitution is mush easier since in this case no reinducing is required as the first return map 
$(f_{M_2})^\tau$ is Gibbs Markov: see Remark~\ref{rmk-GMfM2}.
\end{rmk}

After having summarized the main difficulties that occur for the considered wobbly maps,
especially in the case of $f_{M_1}$, we turn to more details on the method of proof.  
In the particular setup of $f_{M_2}$ the pull back method of Melbourne and T\"or\"ok mentioned in item d)
above (see also~\cite{Gou15} for a summary of the method adapted to stable laws) works with moderate
adjustments to the arguments.
However, in the case of $f_{M_1}$, which requires re-inducing,
it becomes more difficult to prove that the sequences $(n_r)$
along which the scaled ergodic sums of the original map converge are subsequences of
the exponential sequence $(k_n) = (\lfloor c^n \rfloor)$.
Given the result in~\cite{ Gouezel05} along with several estimates in~\cite{BT18},
we found it more convenient to adapt the arguments~\cite{ Gouezel05} to the setup of (tails of) of both $f_{M_1}$  and $f_{M_2}$. We refer to the abstract Theorem~\ref{prop-main} in Section~\ref{sec:abstra}.


An immediate notable consequence of Theorem~\ref{thm.singmap} are upper and lower bounds on the correlation decay. 
For a precise statement on these bounds we need to introduce 
additional notation.
We recall the one in~\cite{BT18}, which is adequate to the present setup but note that this can be 
further refined as in the work of Bruin, Melbourne and Terhesiu~\cite{BMTmaps} (which treats much more complicated classes of dynamical systems) w.r.t.\ the  condition on the seminorm.

As in~\cite{BT18} we consider the following class of observables. Let $X=[0,1]$ and $\tau^*(x) := 1+\min\{ i \geq 0 : f^i(x) \in Y\}$, (here $f$ is either
$f_{M_1}$ or $f_{M_2}$).
Let $s(x, x')$ be the separation time of points $x,x'\in Y$ and let $\theta \in (0,1)$ be such that ~\eqref{eq:locLip} below holds; equation~\eqref{eq:locLip} in Section~\ref{sec-frame} is used to define a Gibbs Markov map.
Let $v_X:X\to\R$. For $\eps > 0$ we define the weighted norm $\|\, \|_\theta^*$ as follows: 
\begin{equation}~\label{eq-extracond2}
\begin{cases}
\|v_X\|_{\infty}^* := \sup_{x\in X}|v_X(x)|\tau^*(x)^{1+\epsilon}, \\[2mm]
|v_X |_\theta^* = \sup_{a \in \alpha} \sup_{0 \leq i < \varphi(a)} \sup_{x,x' \in a}
 \frac{(\tau^* \circ f^i(a))^{1+\eps}}{\theta^{s(x,x')}} |v_X \circ f^i(x) - v_X \circ f^i(x')|,
\end{cases}
\end{equation}
and $\| v_X \|_\theta^* = \| v_X \|_\infty^* + | v_X |_\theta^*$. 
As clarified in Subsection~\ref{subsec-ver} (for $f\equiv f_{M_1}$), 
$\tau^*$ is constant on $f^i(a)$, $0\leq i < \varphi(a)$,
so the factor $\tau^* \circ f^i(a)$ in \eqref{eq-extracond2} 
is well-defined. We note that if $v_X$ is supported on $Y$, then the weighted norms $\|\,\|_{\infty}^*$ and $\| \, \|_\theta^*$ coincide with 
$\|\,\|_{L^\infty(\mu_Y)}$ and $\| \, \|_\theta$ with $\|v\|_\theta=\|v\|_{L^\infty(\mu_Y)}+|v|_\theta$, where $v|_\theta$ is
the H{\"o}lder constant of $v$ w.r.t.\ the distance $d_\theta(x,x'):=\theta^{s(x,x')}$. 

\begin{cor}{~\cite[Theorem 4.2]{BT18}}
\label{lemma-corel-finite}
Assume the setup of $f_{M_1}$ with $M_1$ as in~\eqref{eq.intermittent2} and $\alpha\in (0,1)$.
Suppose that $v, w:[0,1]\to\R$ are such that $\|v\|_\theta^*<\infty$ and $\|w\|_{\infty}^*<\infty$.
Let $d\mu=\frac{1}{\bar\varphi}d\mu$. Then
\begin{align*}
\int v \, w \circ f_{M_1}^{n} \, d\mu -\int v\, d\mu \int w  \, d\mu
=\frac{1}{\bar\varphi}\sum_{j=n+1}^\infty \mu_0(\varphi>j)  \int v\, d\mu \int w  \, d\mu+ E_n,
\end{align*}
where $|E_n |\le C \|v \|_{\theta}^*\, \| w \|_{\infty}^*\, d_n$, for some $C>0$ and 
$$
d_n :=  \begin{cases}
n^{-1/\alpha} &\text{ if } \alpha< 2; \\
n^{-2} \log n &\text{ if } \alpha = 1/2; \\
 n^{-(2/\alpha-2)} &\text{ if } \alpha>2.
\end{cases}
$$
\end{cor}

\paragraph{Notation}
For $a_n,\,b_n>0$ we write $a_n=O(b_n)$ or $a_n \ll b_n$
if there is a constant $C>0$ such that $a_n/b_n\le C$ for all $n\ge1$.
We write $a_n \sim b_n$ if $\lim_n a_n/b_n = 1$,  $a_n=o(b_n)$ if $\lim_n a_n/b_n = 0$
and $a_n\approx b_n$ if $C^{-1}\le a_n/b_n\le C$ for some $C>0$ and all $n\ge1$.

\section{General background on semistable laws}
\label{sec-sst}

Semistable laws are limits of centred and normed sums of i.i.d.\  random variables along 
subsequences $k_n$ for which
\begin{equation}
\label{eq:H}
k_{n}<k_{n+1}, n\geq 1 \text{ and } \lim_{n\to\infty}\frac{k_{n+1}}{k_n}=c > 1. 
\end{equation}
Since $c=1$ corresponds to the stable case (\cite[Theorem 2]{M}), we recall the background in the case $c > 1$. The simplest such a sequence is $k_n = \lfloor c^n \rfloor$.

Let $X, X_1, X_2, \ldots $ be i.i.d.\  random variables with distribution 
function $F(x) = \P ( X \leq x )$. 
Given a semistable random variable $V$ with distribution function $G(x) = \P ( V \leq x)$,
as in~\cite{M}, we say that the random variable $X$ belongs to the domain of 
geometric partial attraction of the semistable law $G$ with index $\beta\in (0,2)$ if there is a subsequence $k_n$ 
for which (\ref{eq:H}) holds, and a norming and  a centring sequence $A_n, B_n$, such 
that
\begin{equation} \label{eq-dgp-conv}
\frac{\sum_{i=1}^{k_n} X_i }{A_{k_n}} - B_{k_n} \rightarrow^d V.
\end{equation}Without loss of generality we may assume that $A_n = n^{1/\beta} \ell_1(n)$, $\beta\in (0,2)$,
with some slowly varying function $\ell_1$ (see \cite[Theorem 3]{M}). Also, $B_{k_n}$ is so that $\lim_{r\to\infty}B_{k_n}/k_n$ exists.
For the form and properties  of the characteristic function of the random variable $V$ we refer to~\cite{Kruglov, Csorgo07}.

In order to characterise the domain of geometric partial attraction we need some further 
definitions. As $k_{n+1} / k_n \to c > 1$, for any 
$x$ large enough there is a unique $k_n$ such that $A_{k_n} \leq x < A_{k_{n+1}}$. 

Let $\delta(x) = \frac{x}{A_{k_n}}$ and note that the definition of $\delta$ does depend on the norming sequence.
Let  $x^{-\beta} \ell(x) : = \sup \{ t : t^{-1/\beta} \ell_1(1/t) > x \}$
and note that  $x^{1/\beta} \ell_1(x)$ and $y^\beta / \ell(y)$ are asymptotic inverses of each 
other.
For properties of asymptotic inverse of regularly varying functions we refer to
\cite[Section 1.7]{BGT}.

By Corollary 3 in \cite{M}, \eqref{eq-dgp-conv} holds on the 
subsequence $k_n$ with norming sequence $A_{k_n}$ if and only if 
\begin{equation} \label{eq-dgp-distf}
\overline F(x) := 1 - F(x) = \frac{\ell(x)}{x^\beta} [ M(\delta(x)) + h(x) ],
\end{equation}
where $M$ is a log periodic function with period $c^{1/\beta}$ with $c$ given in~\eqref{eq:H} (that is, $M(c^{1/\beta}x)=M(x)$)
and  $h$ is right-continuous error function such that
$\lim_{n \to \infty} h(A_{k_n} x ) = 0$, whenever $x$ is a continuity point of $M$.
Moreover, if $M$ is continuous, then $\lim_{x \to \infty} h(x) = 0$. With the parameters $\beta$ and $c$ specified,  
throughout the rest of the paper, $V$ in \eqref{eq-dgp-conv}
will be referred to as $V_{\beta}(c)$ and we mean a random variable distributed according to a semistable distribution of index $\beta$
and parameter $c$. As recalled below, on different subsequences there are different limit distributions belonging to the class of semistable law.

We say that $u_n$ \textit{converges circularly} to 
$u \in (c^{-1},1]$, $u_n \stackrel{cir}{\to} u$, if $u \in (c^{-1},1)$ and $u_n \to u$ 
in the usual sense, or $u =1$ and $u_n$ has limit points $1$, or $c^{-1}$, or both.

For $x > 0$ (large)  define 
\begin{equation} \label{eq:def-delta}
\gamma_x=\frac{x}{k_n}, \quad \text{where } k_{n-1} <  x \leq k_n.
\end{equation}
and note that by \eqref{eq:H}
\[
 c^{-1} =\liminf_{x \to \infty} \gamma_x < \limsup_{x \to \infty } \gamma_x = 1.
\]

The definitions of the parameter $\gamma_n$ and the circular convergence 
follow the definitions in \cite[p.~774 and 776]{KC}, and are slightly different from 
those in \cite{M}.

It follows from Theorem 1 \cite{CM1}  that (\ref{eq-dgp-conv}) holds along a subsequence 
$(n_r)_{r=1}^\infty$ (instead of $k_n$) if and only if 
$\gamma_{n_r} \stackrel{cir}{\to} \lambda \in (c^{-1}, 1]$ as $r \to \infty$. More precisely, whenever $\gamma_{n_r} \stackrel{cir}{\to} \lambda$,
\begin{equation} \label{eq-conv-subseq}
\frac{\sum_{i=1}^{n_r} X_i-B_{n_r}} {n_r^{1/\beta} \ell_1(n_r)} \to^d V_{\lambda}
\quad \text{as } r \to \infty,
\end{equation}
where $B_{n_r}$ is so that $\lim_{r\to\infty}B_{n_r}/n_r$ exists
and 
$V_\lambda$ is a semistable random variable with distribution and characteristic function depending on $\lambda$. For the precise form of the characteristic
function of $V_\lambda$ we refer to~\cite{CM1, PK1}.

\section{Abstract setup}
\label{sec:abstra}

\subsection{Main assumptions}
\label{sec-frame}

Let $f:X\to X$, for $X$ a  Polish space.  
We require that there exists $Y\subset X$ and a \emph{general} (not necessarily first) return time $\varphi: Y\to \N$ such that
the return map $F:=f^\varphi:Y \to Y$  preserving the measure $\mu_Y$ is Gibbs Markov, as recalled below (see \cite[Chapter 4]{Aaronson} for further details).
For convenience we rescale such that $\mu_Y(Y)=1$  and assume that $F$ has a Markov partition $\alpha = \{ a \}$ such that
$\varphi|_a$ is constant on each partition element so that
$F:a \to Y$ is a bijection $\bmod{\, \mu_Y}$. Throughout, we assume that that the greatest common divisor $\text{gcd}( \varphi(a) , a \in \alpha) = 1$.

Let $\alpha_n$, $n\ge 1$ be the $n$-th refined partition  and let $p_n = \log \frac{d\mu}{d\mu\circ F^n}$ be the corresponding potential.
We assume that there is $\theta \in (0,1)$ and 
$C_1 > 0$ such that for all $n\ge 1$,
\begin{equation}\label{eq:locLip}
e^{p_n(y)} \leq C_1 \mu_Y(a),\qquad |e^{p_n(y)}  - e^{p_n(y')}| \leq C_1 \mu_Y(a) \theta^{s(F^ny, F^ny')}, 
\end{equation}
for all  $y, y' \in a,\,  a \in \alpha_n$,
where $s(y_1, y_2) = \min\{ m \geq 0 : F^m(y_1)\text{ and } F^m(y_2) $  belong to different elements of  $\alpha \}$
is the {\em separation time}. We also assume that $\inf_{a \in \alpha}\mu_Y(Fa)>0$ (big image property).

Throughout we let $\tau:Y\to\N$ be the first return time to $Y$ and
write $\{ \tau> n\}: = \{ y \in Y : \tau(y) > n\}$ and the same for $\{ \varphi> n\}$.
We assume that
\begin{itemize}
\item[\textbf{(H0)}] $\mu_Y(\varphi>n)=O(n^{-\beta}\ell(n))$ 
for $\beta> 1$ and some slowly varying function $\ell$. 
\end{itemize}


We define $\varphi$ in terms of consecutive first returns to $Y$.  Since $\tau_k$ is the
$k$-th return time to $Y$, \ie $\tau_0 = 0$,
$\tau_{k+1}(y) = \tau_k(y) + \tau(f^{\tau_k(y)}(y))$ and we let $\rho$ be the {\em reinduced  time} for
the general return, \ie $\varphi(y) = \tau_{\rho(y)}(y)$.
As in~\cite{BT18}, lower bounds on the correlation decay require that
\begin{itemize}
\item[\textbf{(H1)}] $\int_{\{\varphi>n\}}\rho\, d\mu_Y=O(\mu_Y(\varphi>n))$.
\end{itemize}
Assumption  (H1) will be used for simpler arguments in Subsection~\ref{subsec:rellimth}.

\begin{rmk}\label{rmk-expansion}
The first return time $\tau$ may be defined on a
larger set than where the general return time
$\varphi$ is defined, but in our main example the difference in domains has
measure zero, so we will ignore it.
\end{rmk}

For the purpose of simpler arguments in Subsection~\ref{subsec:rellimth}, we also require the following mild condition on the inducing scheme.
\begin{itemize}
\item[\textbf{(H2)}]
Either $f^i(a) \subset Y$ or
$f^i(a) \cap Y = \emptyset$ for all $a \in \alpha, 0 \leq i < \varphi(a)$, 
\end{itemize}

In order to obtain a non- Gaussian limit law for $f_\Delta$ (and thus for $f$),
we assume that $\tau$ behaves according to~\eqref {eq-dgp-distf}.
\begin{itemize} 
\item[\textbf{(H3)}] Suppose that (H0)  holds with $\beta\in (1,2)$.  Assume that 

\begin{itemize}
\item[i)]there exists a sequence $k_n\to\infty$ so that $\lim_{n\to\infty}\frac{k_{n+1}}{k_n}=c> 1$;
\item[ii)] there exists  a right continuous, logarithmically periodic function  $M: (0,\infty) \to (0, \infty)$ with period 
$c^{1/\beta}$, i.e.~$M(c^{1/\beta} x ) = M(x)$ for all $x > 0$.

\item[iii)]Under i)-ii),  suppose that for some slowly varying function $\ell$,
\[
\mu_Y(\tau>x)= x^{-\beta}\ell(x)(M(\delta(x)) + h(x)),
\] where:
\begin{itemize}
\item[$\bullet$)] $\delta(x)=\frac{x}{A_{k_n}}$, for  $x\in [A_{k_n}, A_{k_{n+1}})$
with $A_{k_n}$ defined by $A_{n}:=n^{1/\beta}\ell_1(n)$, where $\ell_1$ is slowly varying 
so that $n^{1/\beta}\ell_1(n)$ is the asymptotic inverse of $n^\beta\ell(n)^{-1}$.

\item[$\bullet\bullet$)]  $h$ is some right-continuous function  such that
$\lim_{n \to \infty} h(A_{k_n} x ) = 0$, whenever $x$ is a continuity point of $M$.
\end{itemize}
\end{itemize}
\end{itemize}
\begin{rmk} The case $c=1$ in (H3) corresponds to $\beta$-stable laws for $f$; it is well understood
via the method in~\cite{MTor} (see, for instance,~\cite{Sarig01})
and omitted here.
\end{rmk}

\begin{itemize} 
\item[\textbf{(H4)}] Suppose that (H3) holds and assume that $\lim_{n\to\infty}\frac{\mu_Y(\varphi>n)}{\mu_Y(\tau>n)}=C$, for some $C>0$.
\end{itemize}

\subsection{Observables on $X$ and $\Delta$ }
\label{sec-towers}

The tower $\Delta$ is the disjoint union 
of sets $(\{ \varphi = j\} ,i)$, $j \geq 1$, $0 \leq i < j$
with tower map
\[
f_\Delta(y,i) = \begin{cases}
(y,i+1) & \text{ if } 0 \leq i < \varphi(y)-1,\\
(F(y),0) & \text{ if } i = \varphi(y)-1.
\end{cases}
\]
This map preserves the measure $\mu_\Delta$ defined as
$\mu_\Delta(A,i) = \mu_Y(A)$ for every measurable set $A$, with $A \subset \{ \varphi = j\}$ and $0 \leq i < j$.

Let $Y_i = \{ (y,i) : \varphi(y) > i\}$ be the $i$-th level of the tower,
so $Y   = Y_0$ is the base.
The restriction $\mu_\Delta|_{Y} = \mu_Y$ is invariant under 
$f_\Delta^\varphi$, which is the first return map to the base.
The function $\varphi$ extends to the tower via $\varphi_\Delta(y,i) := \varphi(y) - i$.

Define $\pi:\Delta \mapsto X$ by $\pi(y,i) := f^i(y)$.
The measure $\mu_X = \mu_\Delta \circ \pi^{-1}$ is $f$-invariant, and $\mu_X$ is related 
to the $F$-invariant measure $\mu_Y$ via 
$\mu_X(A)= \sum_{j=0}^\infty \mu_Y(f^{-j}A \cap \{ \varphi > j \} )$.
Regardless of whether $\bar\varphi := \int_Y \varphi\, d \mu_Y$
is finite (in which case we can normalise $\mu_X$) or not, $\mu_Y$
is absolutely continuous w.r.t.\ $\mu_X$.
 
Let $g_X$ be an observable supported on the original space $X$;
it lifts to an observable on the tower $g:= g_X \circ \pi$.
In what follows we use the method in~\cite{Gouezel04b, Gouezel05} to
study limit theorems  for $f$ and observables $g_X$ supported on $X$ via limit theorems via $f_\Delta$ and $g$.

Given  $g= g_X \circ \pi$, let $g_{Y}=\sum_{j=0}^{\varphi-1} g\circ f_\Delta^j$  be its induced version  (to the base of the tower).
We assume
\begin{itemize}
\item[\textbf{(H5)}]
$\mu_Y(|g_Y|>t)=C\mu_Y(\varphi>t)(1+o(1))$ as $t\to\infty$ for some $C>0$.
\end{itemize}

\subsection{Results in the abstract setup}
\label{sec-res}

Our main result in the abstract setup is on limit laws for $f$ and observables supported on $X$.

\begin{thm}
\label{prop-main} Assume (H0)--(H3) with $\beta\in (1,2)$. Let $g_X:X\to\R$  be a bounded observable on $X$ and assume that its induced version $g_Y$ (to the base of the tower)
satifies (H5).

Let $k_n$ be a subsequence satisfying (H3) i) and  let $A_{n}:=n^{1/\beta}\ell_1(n)$ be defined in (H3) iii). Then there exists a subsequence $n_r$  of $k_n$ so that
\begin{equation*}\frac{\sum_{j=0}^{n_r-1}g_X\circ f^{j}-B_{n_r}\int_X g_X\, d\mu}{A_{n_r}}\rightarrow^d W \text{ as } r\to\infty,
\end{equation*}
for some random variable $W$ with a non trivial distribution and some centering sequence  $B_{n_r}$ is so that $\lim_{r\to\infty}B_{n_r}/n_r$ exists.

Moreover, if (H4) holds, the following holds for some centering sequence  $B_k$ is so that $\lim_{n\to\infty}B_k/k$ exists:
\begin{itemize}
\item[i)]  There is a  sequence $k_n$ such that $\frac{\sum_{j=0}^{k_n-1}g_X\circ f^{j}-B_{k_n}\int_X g_X\, d\mu}{A_{k_n}}\rightarrow^d V_{\beta}(c)$, as $n\to\infty$.
\item[ii)]Given $\gamma_x$ as in~\eqref{eq:def-delta}, whenever $\gamma_{n_r} \stackrel{cir}{\to} \lambda\in (c^{-1}, 1]$,
there is a  sequence $n_r$ such that  $\frac{\sum_{i=1}^{n_r} g_X\circ f^i-B_{n_r}\int_X g_X\, d\mu}{A_{n_r}} \to^d V_{\lambda}$, as  $r \to \infty$, where $V_\lambda$ is as in~\eqref{eq-conv-subseq}.
\end{itemize}
\end{thm}

\subsection{Transfer operators for $F$, $f$ and $f_\Delta$}
\label{sec-gentower}

Let $P$ be the transfer operator associated with
the tower map $f_\Delta$ and potential 
$$
p_\Delta(y,i) := \begin{cases}
0 & \text{ if } i < \varphi(y)-1,\\
p(y)& \text{ if } i = \varphi(y)-1.
\end{cases}
$$
Throughout, we let $R:L^1(\mu_Y)\to L^1(\mu_Y)$ be the transfer operator associated with the base 
map $(F=f^\varphi, Y, \mathcal{A}, \mu_Y)$.
Since $F$ is Gibbs Markov, $R$ has good spectral properties in the space $\B$ of bounded piecewise H{\"o}lder functions compactly 
embedded in $L^\infty(\mu_Y)$. The norm on $\B$ is defined by $\| v \| = |v|_\theta + |v|_\infty$,
where $|v|_\theta = \sup_{a \in \mathcal{A}} \sup_{x \neq y \in a} |v(x)-v(y)| / d_\theta(x,y)$,
where $d_\theta(x,y) = \theta^{s(x,y)}$ for some $\theta\in (0,1)$,
and $$s(x,y) = \min\{ n : F^n(x) \text{ and } F^n(y) \text{ are in different elements of } \mathcal{A}\}$$
is the separation time.
We recall that $1$ is an isolated eigenvalue in the spectrum of $R$ and that $\|R(1_{\{\varphi= n\}})\|\ll \mu_Y(\varphi=n)$
 (see~\cite{AaronsonDenker01} and~\cite{Sarig02}). Throughout we let $\Pi$ be the eigenprojection associated with the eigenvalue $1$; in particular, $\Pi v= \int_Y v\, d\mu_Y$.

Let $g:\Delta\to \R$ and let $g_Y$ be its induced version on $Y$. 
Let $s\ge 0$ and define the following transfer operators that describe the general resp.\ the first return
to the base $Y$:
\begin{align*}
T_{n,s}v &:= 1_{Y} P^n(1_{Y}e^{-sg_n}v),\enspace n\ge0; \\
R_{n,s}v &:= 1_{Y} P^n(1_{\{\varphi=n\}}e^{-sg_n}v)=R(1_{\{\varphi= n\}} e^{-sg_Y}v),\enspace n\ge1.
\end{align*}
As in~\cite{Sarig02, Gouezel04,  Gouezel05}, $T_{n,s}=\sum_{j=1}^{n} R_{j,s} \,T_{n-j, s}$. Note that $\sum_n R_{n,0}=R$. 
Also given $R(z, s)v:= R(z^\varphi e^{-s g_Y} v)$ we have $R(z, s)v=\sum_{n\ge 1} R_{n,s} z^n$.
As shown in~\cite{Sarig02},  $\|R_{n,0}\|\ll \mu_Y(\varphi=n)$.

Since $\|R_{n,s}-R_n\|\to 0$ as $s\to 0$ and $\|R_{n,s}\|\ll \mu_Y(\varphi=n)$, ~\cite[Theorem 2.1]{Gouezel04} ensures that: i) $R(z,s)$ is a continuous perturbation of $R(z)$ and 
ii) since $R(z)$ is close to $R(1)$, we have that $R(z,s)$ is close to $R(1,0)$ for $z$ sufficiently close to $1$ and $s$ small enough.
Thus, the operator $R(z,s)$ has an eigenvalue $\lambda(1,0)=1$, which is isolated in the spectrum
of $R(1,0)=R$ and the family of eigenvalue $\lambda(z,s)$ is continuous in a neighborhood of $(1,0)$. 
As a consequence,  there exist $\eps_0>0$ so that $(I-R(z,s))^{-1}$ is well defined for all  $z\in\bar\D\setminus\{1\}$ and $s\le \eps_0$
and $(I-R(z,s))^{-1}=\sum_{n\ge 0} T_{n,s} z^n$. As recalled in Lemma~\ref{lemma-Ggen} below, the asymptotic behavior of $T_{n,s}$ as $n\to\infty$ for $s\le \eps_0$
can be understood as the $n$-th Fourier coefficient of  $(I-R(z,s))^{-1}$.

Following~\cite{Gouezel05}, to understand the behaviour of $P_s^n v= P^n(e^{-sg_n}v)$ (for $s$ small enough) 
when acting on functions supported on the whole of $\Delta$ via 
the behaviour of $T_{n,s}$ (acting on functions supported on $Y$), we need to define several operators that describe the action of 
$P_s^n$ outside $Y$. Define the transfer operators associated with the end resp.\ beginning
of an orbit on the tower as
\begin{align*}
A_{n,s}v & :=  P^n(1_{\{ \varphi> n \} } e^{-sg_n}v ), \quad  n \geq 0,
\\
B_{n,s}v  & := \begin{cases}
1_{Y} P^n(1_{\{ \varphi_\Delta = n  \} \setminus Y }e^{-sg_n} v ), &  n \geq 1, \\
1_{Y}v, & n = 0.
\end{cases}
\end{align*}
Also, define the transfer operator associated with orbits that do not see the base of the tower by
$$
C_{n,s} v :=  P^n(1_{\{ \varphi_\Delta > n \}\setminus Y } e^{-sg_n}v), \quad n \geq 0.
$$
As noticed in~\cite{Gouezel05}, the following equation describes
the relationship between $T_{n,s}$ and $P_s^n$:
\begin{equation}
\label{eq-Pn}
  P_s^n = \sum_{n_1+n_2+n_3=n} A_{n_1,s}T_{n_2,s}B_{n_3,s} + C_{n,s}.
\end{equation}

\subsection{Relating limit laws for $F$ to limit laws for $f_\Delta$}
\label{subsec:rellimth}

\begin{lemma}{~\cite[Theorem 2.1]{Gouezel04b}} 
\label{lemma-Ggen}Suppose that $\sum_{n\ge 1}\mu_Y(\varphi>n)<\infty$ and recall that $\bar\varphi=\int_Y\varphi\, d\mu_Y$.  Assume that $1-\lambda(1,s)=G(s)(1+o(1))$ as $s\to 0$, where $G(s):[0,\infty)\to [0,\infty)$
with $G(s)=0$, for $s=0$ and $G(s)>0$, for $s>0$.
Then there exist $\eps_0>0$ and two functions $q(s)\to 0$ as $s\to 0$, 
$r(n)\to 0$ as $n\to\infty$ so that for all $s\le \eps_0$ and all $n\ge 1$,
\[
\|T_{n,s}-\frac{1}{\bar\varphi}\Big(1-\frac{G(s)}{\bar\varphi}\Big)^n \Pi\|\le q(t)+r(n).
\]
\end{lemma}

\begin{proof}
This follows from~\cite[Theorem 2.1]{Gouezel04b} in the setup of Young towers with summable returns to the base.
As already  recalled in subsection~\ref{sec-gentower}, the hypotheses of ~\cite[Theorem 2.1]{Gouezel04b} are satisfied. In particular,  there exist $\eps_0>0$
so that $(I-R(z,0))$ is ivertible for all $z\in\bar\D\setminus\{1\}$ and for all $s\le \eps_0$.

Although the statement of ~\cite[Theorem 2.1]{Gouezel04b} is in terms 
of complex perturbations $e^{itg}$ as opposed to $e^{-sg}$  as used here and it is slightly different, the invertibility of $1-\lambda(1,s)$ is all that is required
in the proof of~\cite[Theorem 2.1]{Gouezel04b}.  Our assumption on
$1-\lambda(1,s)$ ensures that $1-\lambda(1,s)$ is invertible for all $s\ne 0$. Hence, the argument of ~\cite[Theorem 2.1]{Gouezel04b} applies as summarized below, yielding the conclusion of the present lemma.
We note that it is essential in the  proof of ~\cite[Theorem 2.1]{Gouezel04b} that
the first perturbation $z^\varphi$ is controlled on $\bar\D$; the perturbation $e^{-sg}$ can be controlled for $s\in \mathbb{C}$ or $s\in \R_{+}$ depending on the desired statement.

As in~\cite{Gouezel04} (see also~\cite[Lemma 2.5]{Gouezel04b}), $T_{n,0}=\frac{1}{\bar\varphi}\Pi+ E_n$, where $\|E_n\|\to 0$; in particular, one can fix $N$ and $\eps>0$ so that   $\|E_n\|\le\eps$ for all $n>N$.
It follows from~\cite[Proof of Proposition 2.10]{Gouezel04b} that the function $r(n)$ in the statement of the present lemma is so that
 $r(n)\le \eps+C_0 N G(s)\Big(1-\frac{G(s)}{\bar\varphi}\Big)^{n-N}\le C_1\eps$, for some $C_0, C_1>0$ as $n\to\infty$. Also, 
it follows from~\cite[Proof of Proposition 2.10]{Gouezel04b} that the function $q(t)$ in the statement of the present lemma is so that
$q(s)\le \sum_{j=0}^n \|F_j(s)\|$, where $F_n(s)$ is the $n$-th coefficient of $\Big(\frac{I-R(z,s)}{1-\bar\varphi^{-1}(1-\bar\varphi^{-1}G(s)z)}\Big)^{-1}-\Big(\frac{I-R(z,0)}{1-z}\Big)^{-1}$. By~\cite[Lemma 2.5]{Gouezel04b}, $\sum_n \|F_n(s)\| \le C F(s)$,
for some $C>0$ and $F(s)\to 0$, as $s\to 0$. Hence, $q(s)\to 0$ as $s\to 0$. ~\end{proof}

As in~\cite{Gouezel04, Gouezel04b}) we use  Lemma~\ref{lemma-Ggen} and~\eqref{eq-Pn} to understand the asymptotic behaviour of $\int_\Delta P_s^n 1\, d\mu_\Delta=\E_{\mu_\Delta}(e^{-sg_n})$,
which is the real Laplace transform of the observable $g$. 
Provided that there exists some sequence $a_n\to\infty$ so that $\frac{1}{\bar\varphi}\Big(1-\frac{G(s/a_n)}{\bar\varphi}\Big)^n$ can be related to a real Laplace transform $\hat W(s)$
of a non trivial distribution $W$, we will infer that $\mu_Y(\frac{g_n}{a_n}<a)$ can be related to $\P(X<a)$, where $X$ is distributed according to $W$.

We recall the estimates below to be used in obtaining the asymptotic of $\int_\Delta P_s^n 1\, d\mu_\Delta$ (as in Proposition~\ref{prop-lthto} below).
\begin{lemma}
\label{lemma-BCn} Let $v:\Delta\to\R$ so that $v\in L^\infty(\mu_\Delta)$. There exists  $C>0$ so that for all $n\geq 0$ and for all $s\ge 0$, $|\int_\Delta C_{n,s} v\, d\mu_\Delta|\leq C \mu_Y(\varphi>n)\|v\|_{L^\infty(\mu_\Delta)}$.

Assume (H1) and suppose that $v$ is H{\"o}lder on each level of $\Delta$ with $v\in L^\infty(\mu_\Delta)$. There exists  $C>0$ so that for all $n\geq 0$ and for all $s\ge 0$,
$\sum_{j>n}\|B_{j,0} (e^{-sg_j}-1)v\|\le C \mu_Y(\varphi>n)\|v\|_{L^\infty(\mu_\Delta)}$. Moreover, $\int_\Delta B(1,0) v\, d\mu_\Delta= \int_\Delta  v\, d\mu_\Delta$.
\end{lemma}
\begin{proof}The statement on $C_{n,s}$ is immediate. For the statement on $B_{n,s}$ under (H1) see, for instance,~\cite[Lemma 6.5]{BT18}.
For a comparable estimate, exploiting $s\to 0$,  without assuming (H1) we refer to~\cite[Lemmas 4.2 and 4.3]{Gouezel05}.For the statement on the integral, see~\cite[Lemma 4.3]{Gouezel05}.~\end{proof}

\begin{lemma}
\label{lemma-An} Let $v, w:\Delta\to\R$  so that 
  $\| (1_Yv) \|_{L^\infty(\mu_Y)}<\infty$, $w\in L^\infty(\mu_\Delta)$.
Assume (H1) and (H2). Then there exists $C>0$ so that for all $n\geq 0$ and for all $s\ge 0$,
\[
\Big|\int_\Delta \sum_{j\geq n} A_{j,0} (e^{-sg_j}-1) (1_Y v)  w\, d\mu_\Delta\Big|\leq C \mu_Y(\varphi>n) \|(1_Y v)\|_{L^\infty(\mu_Y)}\,  \|w\|_{L^\infty(\mu_\Delta)}.
\]
\end{lemma}

\begin{proof}This is a simplified version of~\cite[Lemma 6.2]{BT18}.  For a comparable estimate, exploiting $s\to 0$, without assuming (H1), we refer to~\cite[Lemmas 4.2]{Gouezel05}.~\end{proof}

We can now state

\begin{prop}
\label{prop-lthto} Assume (H1) and (H2).  Assume that $1-\lambda(1,s)=G(s)(1+o(1))$ as $s\to 0$, where $G(s): (0,\infty) \to (0,\infty)$
with $G(s)=0$, for $s=0$ and $G(s)>0$, for $s>0$.
Then there exist $\eps_0>0$ and two functions $q_0(s)\to 0$ as $s\to 0$,  $r_0(n)\to 0$ as $n\to\infty$ so that for all $s\le \eps_0$ and all $n\ge 1$,
\[
\Big|\int_\Delta P_s^n 1\, d\mu_\Delta-\frac{1}{\bar\varphi}\Big(1-\frac{G(s)}{\bar\varphi}\Big)^n \Big|\le q_0(s)+r_0(n).
\]
\end{prop}

\begin{proof} For $z\in\bar\D$, define the operator power series $P(z,s)= \sum_{n\ge 0} P_s^n z^n$ and similarly for
$A(s,z), T(s,z), B(s,z), C(s,z)$. By~\eqref{eq-Pn},  we have $P(z,s)= A(z,s)T(z,s)B(z,s) + C(z,s)$, and so,
\begin{equation*}
\int_\Delta P(z,s) 1\, d\mu_\Delta=\int_\Delta A(z,s)T(z,s)B(z,s)1\, d\mu_\Delta+\int_\Delta  C(z,s)1\, d\mu_\Delta.
\end{equation*}
By Lemmas~\ref{lemma-An} and~\ref{lemma-BCn}, both terms of the RHS are well defined for $z\in\bar\D\setminus\{1\}$ and $s\ge 0$.

Let 
\[
I_A(z,s)=\int_\Delta \frac{A(z, s)-A(1, s)}{z-1} (z-1)T(z,s)B(z,s)1\, d\mu_\Delta,
\]
\[
I_B(z,s)=\int_\Delta A(1, s)(z-1)T(z,s)\frac{B(z,s)-B(1,s)}{z-1}1\, d\mu_\Delta,
\]
and 
\[
I_T(z,s)=
\int_\Delta A(1, s)(z-1)T(z,s)B(1,s)1\, d\mu_\Delta.
\]
Thus, we can write
\begin{align*}
\int_\Delta P(z,s) 1\, d\mu_\Delta=I_T(z,s)+I_A(z,s)+I_B(z,s)+\int_\Delta  C(z,s)1\, d\mu_\Delta.
\end{align*}
By Lemma~\ref{lemma-BCn},  the $n$-th coefficient of $\int_\Delta  C(z,s)1\, d\mu_\Delta$ is $O(\mu_Y(\varphi>n))$.  
By  Lemma~\ref{lemma-Ggen}, $(z-1)T(z,s) v=\frac{1}{\bar\varphi}\sum_n \Big(1-\frac{G(s)}{\bar\varphi}\Big)^n z^n \Pi v + \sum_n (q(s)+r(n)) z^n$,
where $\|q(s)\|\to 0$ as $s\to 0$ and  $\|r(n)\|\to 0$ as $n\to\infty$.   Using this, we estimate the $n$-th coefficients of 
$I_T(z,s)$, $I_A(z,s)$  and $I_B(z,s)$.

Note that $\int_\Delta A(1, s)\Pi B(1,s)1\, d\mu_\Delta=\int_\Delta A(1, s)1_Y\, d\mu_\Delta\, \int_\Delta B(1, s)1\, d\mu_\Delta$. 
By Lemma~\ref{lemma-BCn}  and lemma~\ref{lemma-An}, as $s\to 0$,
$\int_\Delta (A(1,s)-A(1,0))1\, d\mu_\Delta \to 0$
and $\int_\Delta (B(1,s)-B(1,1))1\, d\mu_\Delta \to 0$. Thus, 
$K(s):=\int_\Delta (A(1,s)-A(1,0))1_Y\, d\mu_\Delta\,\int_\Delta (B(1,s)-B(1,0))1_Y\, d\mu_\Delta\to 0$ as $s\to 0$. Putting these together,
\begin{align*}
I_T(z,s)-\frac{1}{\bar\varphi}\sum_n \Big(1-\frac{G(s)}{\bar\varphi}\Big)^n z^n
&= \frac{1}{\bar\varphi}\sum_n \Big(1-\frac{G(s)}{\bar\varphi}\Big)^n\, K(s) z^n\\
&+\sum_n \Big(\int_\Delta A(1, s)(q(s)+r(n))B(1,s)1\, d\mu_\Delta\Big) z^n.
\end{align*}
Since $K(s)\to 0$ as $s\to 0$,  the $n$-th coefficient of the first term of the RHS of the previous displayed equality  goes to $0$, as $s\to 0$. The coefficient of the second term is bounded by $q(s)+r(n)$,
as desired.

Next, compute that
\begin{align*}
I_A(z,s)&=\int_\Delta \frac{A(z, s)-A(1, s)}{z-1} \Pi(B(z,s)1)\, d\mu_\Delta\\
&+\int_\Delta \frac{A(z, s)-A(1, s)}{z-1}\,\sum_n (q(s)+r(n)) z^n\, B(z,s)1\, d\mu_\Delta:= I_A^1(z,s)+  I_A^2(z,s).
\end{align*}
Recall that  $B(z,s)1$ is a function supported on $Y$ and note  that 
\[ I_A^1(z,s)=\int_\Delta \frac{A(z, s)-A(1, s)}{z-1} 1_Y\, d\mu_\Delta\int_\Delta \frac{B(z, s)-B(1, s)}{z-1} 1\, d\mu_\Delta.
\]
By Lemma~\ref{lemma-An}, the $n$-th coefficient of the first factor is $O(\mu_Y(\varphi>n))$. By Lemma~\ref{lemma-BCn}, the $n$-th coefficient of the second  factor is $O(\mu_Y(\varphi>n))$.  Hence, the $n$-th coefficient of
the $I_A^1(z,s)$ is $O(\mu_Y(\varphi>n))$.

 The $n$-th coefficient of $I_A^2(z,s)$ is bounded by the convolution of  $\int_\Delta \frac{A(z, s)-A(1, s)}{z-1} 1_Y\, d\mu_\Delta$
and $\sum_n (q(s)+r(n)) z^n\, B(z,s)1$. Since  $\|B_{n,s} 1\|=O(\mu_Y(\varphi>n))$, the $n$-th coefficient, in norm, of $\sum_n (q(s)+r(n)) z^n\, B(z,s)1$ is $O(n\mu_Y(\varphi>n))=O(n^{-(\beta-1)}\ell(n))$.
The $n$-th coefficient of $I_A^2(z,s)$ is $O(n^{-(\beta-1)}\ell(n))$. 

Finally, 
\begin{align*}
I_B(z,s)&=\int_\Delta A(1,s)\Pi(B(z,s)1)\, d\mu_\Delta
+\int_\Delta A(1,s)\,\sum_n (q(s)+r(n)) z^n\, \frac{B(z,s)-B(1,s)}{z-1}1\, d\mu_\Delta.
\end{align*}
The $n$-th coefficient of first term in the previous displayed equation is $O(\|B_{n,s} 1\|)=O(\mu_Y(\varphi>n))$. By Lemma~\ref{lemma-BCn}, the $n$-th coefficient of the second  term
in the previous displayed equation is $O(n \sum_{j>n}\|B_{j,0} (e^{-sg_j}-1)1\|)=O(n \mu_Y(\varphi>n))= O(n^{-(\beta-1)}\ell(n))$.

The conclusion follows taking $q_0(s)=K(s)+r(s)$ and $r_0(n)=r(n)+O(n^{-(\beta-1)}\ell(n))$.~\end{proof}

\subsection{Proof of Theorem~\ref{prop-main}}

Recall that $\lambda(1,s), s\ge 0$ is the family of eigenvalues for $R(1,s) v=R(e^{-sg_Y} v)$. In order to apply Proposition~\ref{prop-lthto} in the proof of Theorem~\ref{prop-main}, we need to understand the asymptotics 
of $1-\lambda(1,s)$ as $s\to 0$.

Throughout this
section we assume that $g_Y$ satisfies  (H5).
By adding a positive constant to the bounded function $g_X$, we can assume without loss of generality that $g_X$ is a positive and restrict 
the proofs below to the case $g_Y:Y\to\R_{+}$.  With this convention, we state

\begin{lemma} 
\label{lemma-1s}Assume (H0) and (H3). Let $g_X:X\to\R_{+}$ be a bounded obervable  and  suppose that 
$\mu_Y(g_Y>t)=C\mu_Y(\varphi>t)(1+o(1))$ as $t\to\infty$. 

Let $(Z_j)_{j\ge 1}$ be  positive i.i.d.\  random variables so that  $\P (Z_1>t)=\mu_Y(g_Y>t)(1+o(1))$
as $t\to\infty$. Let  $(X_j)_{j\ge 1}$ be positive i.i.d.\  random variables so that $\P(X_1>t)=(\bar\varphi)^{-1}\mu_Y(\tau>t)(1+o(1))$, as $t\to\infty$.

For $s\ge 0$,  set $G(s)=1-\E_{\P}(e^{-s Z_1})$ and $\tilde G(s)=1-\E_{\P}(e^{-s X_1})$. Then 
\begin{itemize}
\item[i)]  for all $s>0$, $\tilde G(s)> 0$ and $\tilde G(s)\le (\bar\varphi)^{-1} G(s)\le C \tilde G(s)$, for some $C\ge 1$;
\item[ii)] $1-\lambda(1,s)=G(s)(1+o(1))$ as $s\to 0$.
\end{itemize}
Moreover, if (H4) holds then $G(s)=C\bar\varphi\, \tilde G(s)(1+o(1))$, for $C>0$ as in (H4).~\end{lemma}

\begin{proof}Let $v(1,s)$ be the family of eigenvectors associated with $\lambda(1,s)$ and let $v(1,0)$ be the normalised eigenvector so that
$\int_Y v(1,0)\, d\mu_Y=1$.
As in~\cite{AaronsonDenker01, Gouezel04b}, write
\[
1-\lambda(1,s)=\int_Y(1-e^{-sg_Y}) v(1,0)\, d\mu_Y+\int_Y(R(1,s)-R(1,0))( v(1,s)-v(1,0))\, d\mu_Y.
\]
Since $g_X$ is bounded, so is $g=g_X\circ \pi$  and thus, $g_Y\ll n$ on $\{\varphi= n\}$.
Since we also know that $\mu_Y(g_Y>t)=C\mu_Y(\varphi>t)(1+o(1))\ll t^{-\beta}$, a direct computation (see, for instance,~\cite[Lemma 4,1]{Gouezel05} with $z=1$ there) shows that 
\begin{align*}
\|R(1,s)-R(1,0)\|\ll \sum_{n}\|R(1_{\{\varphi= n\}} (e^{-sg_Y}-1))\|\ll s^\beta\ell(1/s).
\end{align*}
Hence,
$V(s):=|\int_Y(R(1,s)-R(1,0))( v(1,s)-v(1,0))\, d\mu_Y|\ll s^{2\beta}\ell(1/s)^2$.

Given  that $(Z_j)_{j\ge 1}$  are i.i.d.\  random variables as in the statement of the lemma, 
\[\int_Y(1-e^{-sg_Y}) v(1,0)\, d\mu_Y=1-\E_{\P}(e^{-s Z_1})=G(s).\]
 It remains to estimate $G(s)$.
By (H0) and (H3),  $1\le \frac{\mu_Y(\varphi>n)}{\mu_Y(\tau>n)}\le C$,  for some $C>1$.
Recall that $\P(X_1>n)=(\bar\varphi)^{-1}\mu_Y(\tau>t)(1+o(1))$ and $\tilde G(s)=1-\E_{\P}(e^{-s X_1})$. 
Clearly, $\tilde G(s)\le \bar\varphi^{-1} G(s)\le C \tilde G(s)$
and  $G(s)=C\bar\varphi\, \tilde G(s)(1+o(1))$ if (H4) holds.

Since $X_1$ satisfies (H3),~\cite[Theorem 1]{Kruglov} (see also~\cite[Lemma 1]{Csorgo07})
ensures that there exist $C_2> C_1>0$ so that $C_1 s^{\beta}\ell(1/s)\le \tilde G(s)+s\E_{\P}( X_1)\le C_2s^{\beta}\ell(1/s)$; the analysis in~\cite{Kruglov, Csorgo07} is in terms of
Fourier transforms and carries over with simplified arguments to real Laplace transforms, as required for  $\tilde G(s)$.  
In particular,  given $C_0:=\inf_{x>0} M(x)>0$, we have
\begin{align*}
\tilde G(s)&+s\E_{\P}( X_1)=\int_0^\infty (1- e^{-sx}+sx)\, d\P(X_1>x)>\int_{1/s}^\infty (1- e^{-sx}) \P(X_1>x)\, dx\\
&> C s^\beta \int_1^\infty e^{-\sigma} \sigma^{-\beta}M(\sigma/s)\ell(\sigma/s)\, d\sigma > C C_0 s^\beta\ell(1/s) \int_1^\infty e^{-\sigma}\sigma^{-(\beta-\eps)}\, d\sigma,
\end{align*}
where in the inequality we have used Potter's bounds~\cite{BGT}. Since the integral is well defined, the claim for the $\inf$  follows
and the reverse inequality follows similarly with $\sup$ instead of $\inf$.
It follows that $G(s)> 0$, for all $s>0$ and $1-\lambda(1,s)= G(s) +V(s)=G(s)(1+O(s^{\beta}\ell(1/s))$.~\end{proof}

\begin{pfof}{~Theorem~\ref{prop-main}} Given $k_n$ as in (H3), $g_X$ as in the statement of Theorem~\ref{prop-main} and $g=g_X\circ \pi$, set
 $S_{k_n}(g_X)=\sum_{j=0}^{k_n-1}g_X\circ f^{j}$ and  $S_{k_n}(g)=\sum_{j=0}^{k_n-1}g\circ f_\Delta^{j}$. 

Let $A_{k_n}$ as in (H3) and note that $\E_{\mu_X}(e^{-s A_{k_n}^{-1}S_{k_n}(g_X)}) =\E_{\mu_\Delta}(e^{-s A_{k_n}^{-1}S_{k_n}(g)})$. 
By Lemma~\ref{lemma-1s}, the assumption on $1-\lambda(1,s)$ in statements of Proposition~\ref{prop-lthto} is satisfied.
It follows from Proposition~\ref{prop-lthto} that as $n\to\infty$ and $s\to 0$, 
\begin{equation}
\label{eq-l0}
\E_{\mu_\Delta}(e^{-s A_{k_n}^{-1}S_{k_n}(g)})-\frac{1}{\bar\varphi}\Big(1-\frac{G(s A_{k_n}^{-1})}{\bar\varphi}\Big)^{k_n}\to 0,
\end{equation}
for $G(s)=1-\E_{\P}(e^{-s Z_1})$, where $Z_1$ is distributed like $\varphi$.
By Lemma~\ref{lemma-1s}, 
\begin{equation}
\label{eq-l1}
(1-\tilde G( A_{k_n}^{-1} s)\Big)^{k_n}\le \Big(1-\frac{G(s A_{k_n}^{-1})}{\bar\varphi}\Big)^{k_n}\le C\Big(1-\tilde G( A_{k_n}^{-1}s)\Big)^{k_n},
\end{equation}
for $\tilde G(s)=1-\E_{\P}(e^{-s X_1})$, where $X_1$ is distributed like $\tau$ and for some $C>1$.

 Given $(X_j)_{j\ge 1}$ i.i.d.\  random variables as in the statement of  Lemma~\ref{lemma-1s}, ~\cite[Corollary 3]{M} ensures 
that for a suitable centering sequence $B_{k_n}$,
\[A_{k_n}^{-1}(\sum_{j=0}^{k_n-1}X_i-B_{k_n})\rightarrow^d V_\beta(c),\] where $V_\beta(c)$ is a semistable random variable as recalled in Section~\ref{sec-sst} (in particular, see~\eqref{eq-dgp-distf}).
Multiplying equation \eqref{eq-l1} by $e^{-sB_{k_n}A_{k_n}^{-1}}$, we obtain
\begin{align}
\label{eq-l11111}
\nonumber &e^{-s\frac{B_{k_n}}{A_{k_n}}}\E_{\P}\Big(\exp\Big(-s\frac{\sum_{i=1}^{k_n}X_i-B_{k_n}}{A_{k_n}}\Big)\Big)  \le e^{-s\frac{B_{k_n}}{A_{k_n}}}\Big(1-\frac{G(s A_{k_n}^{-1})}{\bar\varphi}\Big)^{k_n}\\
& \le C e^{-s\frac{B_{k_n}}{A_{k_n}}}\E_{\P}\Big(\exp\Big(-s\frac{\sum_{i=1}^{k_n}X_i-B_{k_n}}{A_{k_n}}\Big)  \Big)
\end{align}

Starting from~\eqref{eq-l11111} and extracting a subsequence $n_r$, we have that,
\begin{equation*}
\frac{\E_{\P}(\exp(-s A_{n_r}^{-1}(\sum_{i=1}^{n_r}Z_i-B_{n_r})))}{\E_{\P}(\exp(-sA_{n_r}^{-1} (\sum_{i=1}^{n_r}X_i-B_{n_r})))}\to b\in (1, C)\text{ as } r\to\infty,
\end{equation*}
for some $b\in (1, C)$.
This together with~\eqref{eq-l0} implies that 
\begin{equation*}
\frac{\E_{\P}(\exp(-s A_{n_r}^{-1}( S_{n_r}(g)-B^*_{n_r})))}{\E_{\P}(\exp(-s\log b\,A_{n_r}^{-1} (\sum_{i=1}^{n_r}X_i-B_{n_r})))}\to 1\text{ as } r\to\infty,
\end{equation*}
where $B^*_{n_r} = B^*_{n_r} -  n_r \log b$.
The previous displayed equation together with the continuity theorem for Laplace transforms (as in~\cite[Chapter XIII, Theorem 2 and Theorem 2(a)]{Feller} yields
\[
\lim_{r\to\infty}\sup_{y>0}\Big| \P(A_{n_r}^{-1} (S_{n_r}(g)-B^*_{n_r})>y)-\P(A_{n_r}^{-1}( \sum_{i=1}^{n_r}X_i-B_{n_r})>y)\Big|\to 0.
\]
Since $X_1$ is distributed like $\tau$ and $\mu(\tau>n)$ satisfies (H3), the tail probability $\P(X_1>n)$ satisfies (H3) along $k_n$. 
The merging result~\cite[Theorem 2]{CM1} ensures that
\[
\lim_{r\to\infty}\sup_{y>0}\Big| \P(A_{n_r}^{-1} (\sum_{i=1}^{n_r}X_i-B_n)>y)-\P(V_{\gamma_{n_r}}^\beta(c)>y)\Big|\to 0,
\]
where the parameter $\gamma_n$ is defined as in~\eqref{eq:def-delta}. Putting together the previous two displayed equations, we obtain
\[
\lim_{r\to\infty}\sup_{y>0}\Big| \P(A_{n_r}^{-1} (S_{n_r}(g)-B^*_{n_r})>y)-\P(V_{\gamma_{n_r}}^\beta(c)>y)\Big|\to 0.
\]
As in ~\cite[Section 1]{CM1}, the distribution $F_n(y)=\P(V_{\gamma_{n}}^\beta(c)<y)$
is stochastically compact; that is, every sequence  contains a further subsequence $r_n$ 
so that $F_{r_n}(y)\to G_W(y)$ as $n\to\infty$, for some 
random variable $W$ with non trivial distribution $G_W$. 
Hence, $A_{n_r}^{-1} (S_{n_r}(g)-B_{n_r})\to^d W$, where $W$ has a non trivial distribution 
 (although we cannot say that $W$ is in the domain of a semistable law).

For the statement under (H4), we just need to recall that from Lemma~\ref{lemma-1s} that $G(s)=C\bar\varphi\, \tilde G(s)(1+o(1))$, for $C>0$ as in (H4). In this case, 
~\cite[Corollary 3]{M} applies directly to $(Z_j)_{j\ge 1}$ as in the statement of Lemma~\ref{lemma-1s} ensuring that 
\[
A_{k_n}^{-1}( \sum_{i=1}^{k_n}Z_i-B_{k_n})
\]
converges in distribution with respect to the measure $\mu_Y$. This together with~\eqref{eq-l0} gives the conclusion for the convergence in distribution with respect to the measure $\mu_Y$
of $ A_{k_n}^{-1}(S_{k_n}(g)-B_{k_n}^1)$. The strong distributional convergence 
$ A_{k_n}^{-1}(S_{k_n}(g)-B_{k_n}^1)\to^d V_{\beta}(c)$ follows from this together with~\cite[Proposition 4.1]{ThalerZweimuller06}. ~\end{pfof}

\section{On the asymptotics of the tail of first return times $\tau$ 
for $f_{M_1}$ and $f_{M_2}$, and verification of (H3)}\label{sec.first-tail}

Recall the wobbly maps $f_{M_1}$, $f_{M_2}$  with $M_1, M_2$ introduced in in Subsection~\ref{sec-infst}.
In this section we prove the following tail estimates in the setup of $f_{M_1}$ for the first return time
$\tau(x)=\inf \{ n \geq 1 : f_{M_1}^n(x) \in Y\}$, $x\in Y$, and $Y=[1/2,1]$.
The verification in the setup of $f_{M_2}$ is similar.
We show that there are functions 
$\tilde{M},h,\delta$ and $\ell$ satisfying the hypothesis of Corollary 2 in \cite{M} with
\begin{equation}
    \label{eq:tail-of-tau-sin-final}
    m_Y( \tau > y ) = y^{-\beta} \ell(y) ( \tilde{M} ( \delta(y) ) + h(y)),
\text{ where } \beta=1/\alpha.
\end{equation}
In turn this allows us to verify condition (H3) for  $f_{M_1}$ along $k_n=\lfloor e^{n \alpha c_1}\rfloor$. 
A similar analysis holds for $f_{M_2}$ (with $c_2$ instead of $c_1$)
with a similar (but simpler) proof and we omit this.

To ease the notation, throughout this section we write $f$ instead of  $f_{M_1}$

Define below a monotone sequence $x_n\to 0$, 
such that $f^k(x_n)\in(0,1/2)$ for all $k\leq n$, and $f^{n}(x_n)=1/2$. 
This sequence $(x_n)$ will allow us to study the set $\{x \in Y:\tau(x)=n\}$ for $Y = (\frac12,1]$.
First consider the local behaviour of $f$ near the critical set $\mathcal{C}=\{s_{\ell}
=e^{-\ell c_1},\ell \geq 1\}$.
If $n\leq x< n+1$, then $\{x\}=x-n$. For $\epsilon$ small, consider the values
$M(u^{\pm}_{\epsilon}),$ with $u^{\pm}_{\epsilon}=e^{-(\ell\pm\epsilon) c_1}$.
We have:
\begin{equation*}
M(u^{+}_{\epsilon})=C_02^{-\{\ell+\epsilon\}}=C_02^{-\epsilon},\quad
M(u^{-}_{\epsilon})=C_02^{-\{\ell-\epsilon\}}=C_02^{-1+\epsilon}.
\end{equation*}
At $\mathcal{C}$, let $s^{\pm}_{\ell}$ denote the usual upper $(+)$ and lower $(-)$ approaching limits to $s_{\ell}=e^{-\ell c_1}$.
Then:
\begin{equation*}
f(s^{+}_{\ell})=s_{\ell}+C_0s_{\ell}^{1+\alpha},\quad
f(s^{-}_{\ell})=s_{\ell}+\frac{C_0}{2}s_{\ell}^{1+\alpha}.
\end{equation*}
Hence $f(s^{+}_{\ell})>f(s^{-}_{\ell})$, with jump size equal to 
$\frac{C_0}{2}s_{\ell}^{1+\alpha}$. 

Let $x_0=1/2$, and define the sequence $x_n\to 0$, with the property that $f^k(x_n)\in(0,1/2)$ 
for all $k\leq n$, and $f^{n}(x_n)=1/2$.
Due to the jumps in $f$ at $\mathcal{C}$, the inverse map $f^{-1}(x)$ is not always defined, and so we cannot
immediately set $x_n=f^{-n}(1/2)$. We begin (inductively) by setting $x_{n+1}=f^{-1}(x_n)$ in the case
for which $x_{n+1}$ is defined. This will certainly apply for small values of $n$. 
An issue arises in the case 
$x_n\in[f(s^{-}_{\ell}),f(s^{+}_{\ell})]$ for some $s_{\ell}\in\mathcal{C}$. 
In this case we set $x_{n+1}=s_{\ell}$,
and then continue iterating backwards. If $x_n\not\in[f(s^{-}_{\ell}),f(s^{+}_{\ell})]$ 
then again no issue arises and we set $x_{n+1}=f^{-1}(x_n)$. 
For this sequence we define $J_n=[x_{n+1},x_n]$. By construction
we have $\{\tau>n+1\}=f^{-1}([0,x_n]) \cap Y$. 

We have the following proposition concerning
the asymptotics of the sequence $(x_n)$. 
\begin{prop}
    \label{prop:tail-of-xn-sin}
    Consider the map $f$ defined by 
\eqref{eq.intermittent2}. The sequence $x_n$ satisfies
    \begin{equation}\label{eq:tail-of-x_n}
        x_n=\frac{1}{  (\alpha M_0(n)n)^{\beta} }
+O\left(\frac{\log n}{n^{1+\beta}}\right),
    \end{equation}
with $\beta=1/\alpha$, and
        \begin{equation}
            \label{eq:M_0-log}
            M_{0}(n) = \frac{1}{n}\sum_{j=1}^{n}M(x_j).
        \end{equation}
\end{prop}
\begin{rmk}
Notice that $M_0(n)$ is the average of $M(x)$ along the sequence $(x_j)$
up to iterate $n$. It is clear that $M_0(n)\in[C_0/2, C_0]$.
\end{rmk}

\begin{proof}
In the case where $x_n\not\in[f(s^{-}_{\ell}),f(s^{+}_{\ell}))$
then $x_{n+1}$ is derived from $x_n=x_{n+1}+M(x_{n+1})x^{1+\alpha}_{n+1}$.
If $x_n\in(f(s^{-}_{\ell}),f(s^{+}_{\ell}))$, then we set $s_{\ell}=x_{n+1}$
and this gives rise to an (absolute) error $\xi_n$ given by:
$$\xi_n=\left|\frac{x_n-x_{n+1}}{x^{1+\alpha}_{n+1}}-M(x_{n+1})\right|,$$
and this is bounded by 
$$|M(s^{+}_{\ell})-M(s^{-}_{\ell})|s_{\ell}^{1+\alpha}\leq \frac{C_0}{2}
s_{\ell}^{1+\alpha}.$$
Define
$$\widehat{M}_0(n)=M_0(n)+\frac{1}{n}\sum_{j=1}^{n}\xi_j,$$
with $\xi_j=0$ in the case  $x_j\not\in(f(s^{-}_{\ell}),f(s^{+}_{\ell}))$.
Then we have
$$x_n=x_{n+1}+[M(x_{n+1})+\xi_{n+1}]x^{1+\alpha}_{n+1}.$$
To find the asymptotics, let $u_n=x^{-\alpha}_n$, and set
$\widehat{M}(x_n)=M(x_{n})+\xi_{n}.$
Then 
\begin{equation}
\begin{split}
u_n &=u_{n+1}\left(1+\frac{\widehat{M}(x_n)}{u_{n+1}}\right)^{-\alpha}\\
&=u_{n+1}-\alpha\widehat{M}(x_n)+\alpha(\alpha+1)\frac{\widehat{M}(x_n)^2}{2u_{n+1}}(1+o(1)).\\
\end{split}
\end{equation}
This leads to the asymptotic expansion:
\begin{equation}\label{eq.un-asymp}
\begin{split}
u_n &=u_0+\sum_{j=0}^n \widehat{M}(x_j)\ 
+\alpha\sum_{j=0}^n\frac{\alpha(\alpha+1)\widehat{M}(x_j)^2}{2u_0+2\sum_{l=0}^j \widehat{M}(x_l) }(1+o(1))\\
&=u_0+\alpha n \widehat{M}_0(n)+ \sum_{j=0}^n \frac{ \widehat{M}^2(u_j)}{u_0+
j\widehat{M}_0(j) }(1+o(1)).
\end{split}
\end{equation}
Since $\widehat{M}(x_n)$ is bounded, it follows
that $u_n\in[\frac{n}{C_1},C_1n],$ for some $C_1>0$ depending
only on $\alpha$ and $c_1$.
Hence $x_n\in[\left(\frac{1}{C_1n}\right)^{\beta},
\left(\frac{C_1}{n}\right)^{\beta}]$.
These are rough bounds which we can now improve on. We claim
that
$$\widehat{M}_0(n)=M_0(n)+O\left(\frac{\log n}{n}\right).$$
To show this we consider integers $j\leq n$ for which $\xi_j\neq 0$. 
For such $j$, there exists $\ell\in\mathbb{N}$ with
$x_{j}<e^{-\ell c_1}<x_{j-1}$.
We claim that there can be no more than $2c^{-1}_1\log n$ 
such integers $j\leq n$.
Indeed if there are $k$ such values of $j$, namely $j_1,\ldots, j_k$,
then we must have $x_{j_k}<e^{-k c_1}$. This follows from the fact
that the sequence $(x_{n})$ is strictly monotone decreasing, and if $j_i$ is such that
$x_{j_i}<e^{-\ell_i c_1}<x_{j_i-1}$, 
then we must have $x_{j_{i+1}}<e^{-(\ell_{i}+1)c_1}$, i.e. $\ell_{i+1}>\ell_{i}+1$.  
However $x_n\in[(\frac{1}{C_1n})^{\frac{1}{\alpha}},(\frac{C_1}{n})^{\frac{1}{\alpha}}]$, 
and hence $k\leq \beta c^{-1}_1\log n+C_2$,
where $C_2>0$ is a uniform constant. The claim follows.

Returning to the asymptotic expansion for $u_n$ in equation \eqref{eq.un-asymp},
we now have the refinement:
\begin{equation}
u_n=u_0+\alpha nM_0(n)+O(\log n), 
\end{equation}
where the third term on the right hand side of \eqref{eq.un-asymp} is also
$O(\log n)$ (via a harmonic series bound). Inverting for $x_n$ gives:
\begin{equation}
x_n=\left(u_0+\alpha nM_0(n)+O(\log n)\right)^{-\frac{1}{\alpha}}
=\frac{1}{(\alpha nM_0(n))^{\frac{1}{\alpha}}}+
O\left(\frac{\log n}{n^{\frac{1}{\alpha}+1}}\right).
\end{equation}
This completes the proof.

\end{proof}

Here and throughout let us extend $M_{0}$ to the positive reals so that
$M_{0}(x):= M_{0}(\floor{x})$. We shall write 
$x_n=(\alpha nM_0(n))^{-1/\alpha}+n^{-1/\alpha}E_0(n)$, with
$E_0(n)=o(1)$. We extend $E(x)$ to $x\in [0,\infty)$ via $E_{0}(x):= E_{0}(\floor{x})$.
We now examine further properties of the function $M_0$. To account for
different versions of $M(x)$, we keep the argument general. To this end,
let $p: \R \to \R$ be a piecewise Lipschitz periodic function of period $c_1$.
Define $M(x) = a + bp(\log(x))$ where $a > 0$ and $b$ is so small that $M(x)$ 
is bounded and bounded
away from zero. Note that $M$ is log-periodic with period $e^{c_1}$.
Recall that $M_0(n) = \frac1n \sum_{j=1}^n M(x_j)$ is bounded and bounded away from zero.
We have the following result:

\begin{prop}\label{prop-4.3}Set $c=e^{\alpha c_1}$ and  $w_k=\lfloor c^k\rfloor$.
Then there exists $\zeta \in \R$ such that
 \begin{equation}\label{eq-w}
  \lim_{k\to\infty} \frac{w_{k+1}}{w_k} =c  \qquad \text{ and } \qquad  \lim_{k\to\infty} M_0(w_k) = \zeta.
 \end{equation}
\end{prop}

\begin{proof}
The proof is based on the premise that $M_0$ is asymptotically log-periodic,
 where the ``asymptotic'' is faster than any power of the logarithm:
 \begin{equation}\label{eq-asymp-log-periodic}
  \left| \frac{M_0(c n)}{M_0(n)} - 1 \right| \leq \eps_n, \qquad \eps_n = o((\log n)^{-r}),
 \end{equation}
for any fixed $r > 0$. This implies that ~\eqref{eq-w} holds  for some sequence $m_k$ so that $m_k \geq C c^k$ for some $C > 0$ so that $\frac{m_{k+1}}{m_k} \to c$.
In fact we'll show that  ~\eqref{eq-w} holds  for $w_k=\lfloor c^k\rfloor$.
Note that
$$
|M_0(w_{k+1}) - M_0(w_k)| \leq M_0(w_k)  \left| \frac{M_0(c n)}{M_0(n)} - 1 \right| = O(\eps_{w_k})
= O(\eps_{\lfloor c^k \rfloor}).
$$
Because $\eps_n \to 0$ faster than any power of $1/\log n$, the terms $O(\eps_{\lfloor c^k \rfloor})$
are summable in $k$, and therefore $(M_0(w_k))_{k \in \N}$ is a Cauchy sequence, and hence convergent.

The proof therefore relies on proving \eqref{eq-asymp-log-periodic}.
We do this by replacing $M_0(n)$ by the solution of an integral equation.
From equation \eqref{eq:tail-of-x_n}, we have 
$$
x_n=\frac{1}{(\alpha M_0(n)n)^{\beta}}+O\left(\frac{\log n}{n^{1+\beta}}\right)
= \frac{1+\eps_n}{(\alpha M_0(n)n)^{\beta}},
$$
where $e_n = O(n^{-1}\log n)$. 
Hence, using the explicit form for $M(x)$, we have
$$
M_0(n) = \frac1n \sum_{j=1}^n M(x_j)
=\frac{1}{n}\sum_{j=1}^{n} a+b p\Big(\log ((\alpha j M_0(j))^{-\beta}) 
+ \log(1+e_j))\Big).
$$
Since $p$ is Lipschitz, $p(\log( (\alpha j M_0(j))^{-\beta}   )+\log(1+e_j)))-p(\log ( (\alpha j M_0(j))^{-\beta} )=O(e_j)$, 
which gives:
\begin{equation}\label{eq-M0error}
M_0(n) -  M(  (\alpha j M_0(j))^{-\beta}   )  
=\frac{1}{n}\sum_{j=1}^{n} O(e_j)=O\left(\frac{1}{n}\sum_{j=1}^{n} \frac{\log j}{j}\right) = O\left(\frac{(\log n)^2}{n}\right).
\end{equation}
This estimate suggests to replace $M_0(n)$ by a continuous version $\overline{M}(x)$, 
defined implicitly by
\begin{equation}\label{eq-integral}
\overline{M}(x)=\frac{1}{x}\int_1^xM((\alpha u\overline{M}(u))^{-\beta})\, du.
\end{equation}
We show in the Lemma~\ref{lem-1} that $\overline{M}$ is asymptotically log-periodic.
To make the step back to $M_0(n)$, note that
$H(x) := M((\alpha x M_0(x))^{-\beta})$ is not entirely log-periodic with period $c$,
but still has $O(\log x)$ intervals of monotonicity on $[0,x]$.
On each such interval $L$,
$|\int_L H(u) \, du - \sum_{j \in L \cap \Z} H(j) | \leq \sup_L H - \inf_L H$.
Therefore
$$
\int_1^x H(u)\,du =\sum_{j=0}^x H(j) + O(\log x).
$$
Since in our case, $H(x)$ is bounded away from zero, the sum dominates the $O(\log x)$ term.

It follows that if the integral is asymptotically log-periodic, so is the sum:
\begin{eqnarray*}
 \frac{\frac{1}{c x} \sum_{j=1}^{c x} H(j) }{\frac{1}{x} \sum_{j=1}^{x} H(j) }
 &=&
  \frac{\frac{1}{c x} \left( \int_1^{c x} H(j) + O(\log c x) \right) }{\frac{1}{x} 
  \left(\int_1^{x} H(j) + O(\log x)\right) } \\
 &=& \frac{\overline{H}(c) (1 +  \frac{O(\log c x)}{c x \overline{H}(e^cx)} ) }
 {\overline{H}(x)(1 + \frac{O(\log x)}{x \overline{H}(x)} )}\\
 &=& \frac{\overline{H}(c)}{\overline{H}(x)} \left( 1 + \frac{O(\log c x)}{c x \overline{H}(c x)} 
 + \frac{O(\log x)}{x \overline{H}(x)}\right).
 \end{eqnarray*}
This error term is small enough for the purpose of \eqref{eq-asymp-log-periodic}.
\end{proof}

\begin{lemma}\label{lem-1}
 If $M(x) = a(1+b p(\log x))$, for $a > 0$ and $b> 0$ sufficiently small,
 then $\overline{M}$ satisfies \eqref{eq-asymp-log-periodic} with $\eps_n = O(1/n)$.
\end{lemma}

\begin{proof}
We seek an asymptotic solution of this integral
equation \eqref{eq-integral}: we will transform it to a differential equation. First 
let $V(x)=x\overline{M}(x)$. 
Then
$$
V(x)=\int_1^x M((\alpha V(u))^{-\beta})\,du.
$$
Differentiating with respect to $x$ gives:
$$
\frac{dV}{dx}= M((\alpha V(x))^{-\beta})=a(1+bp(-\beta\log V(x)-\beta\log\alpha)).
$$
Now set $U(x) := \log V(x) \sim \log x$. Then
$$
\frac{dU}{dx} e^U =a(1+bp(-\beta U(x)+\log\alpha)).
$$
Separating variables (as far as possible), and integrating gives:
$$
g(U):=\int_{0}^{U}\frac{e^z}{a(1+bp(-\beta z+\alpha_1))}\,dz=x+ C^*,
$$
with $\alpha_1=-\beta\log\alpha$ and $C^{*}$ an integration constant.


Now the left hand integral does not admit a closed form, but we can
study $g(U)$ as follows: given $U$, choose $\ell=\lfloor \beta U/c_1 \rfloor$ and set $\hat c=c_1/\beta=c_1\alpha $ (so given $c$ as in the statement, $c=e^{\hat c}$). Then
$$
g(U)=\sum_{j=0}^{\ell-1}\int_{\hat cj}^{\hat c(j+1)}\frac{e^z}{a(1+bp(-\beta z+\alpha_1))}
\,dz+\int_{\hat c\ell}^{U}\frac{e^z}{a(1+bp(-\beta z+\alpha_1))}\,dz.
$$
For each of the integrals within the sum, let $z'=z-\hat c\,j$ be a new substitution variable.
Then the periodicity of the denominator gives:
\begin{equation}\label{eq.gu}
g(U)=\sum_{j=0}^{\ell-1}e^{j\hat c}g(\hat c)+e^{\ell \hat c}g(U\,\textrm{mod}\,\hat c)
= \frac{e^{\ell \hat c}-1}{e^{\hat c} - 1} g(\hat c) + e^{\ell \hat c} 
g(U\,\textrm{mod}\,\hat c).
\end{equation}
Notice that $g(U\,\textrm{mod}\,\hat c)$
is continuous on $[0, \hat c)$ and bounded by $g(\hat c)$.
Thus the solution is $U(x)=g^{-1}(x+ C^*)$ for some integration constant $C^*$.
Now $\overline{M}(x)=e^{U(x)}x^{-1}$, and we want to show that
$$
\lim_{x\to\infty}\frac{\overline{M}(e^{\hat c}x)}{\overline{M}(x)}=1.
$$
Thus asymptotic periodicity of $\overline{M}(x)$ corresponds
to showing $U(e^{\hat c}x)-U(x)=\hat c+o(1)$ as $x\to\infty$.

Consider the levels $U(x)=\hat c\ell+z$, and $U(x')=\hat c(\ell+1)+z$, with
$\ell$ large and $z\in[0, \hat c]$.  By geometric series, see \eqref{eq.gu}, we have:
$$
x=g(\hat c) \frac{e^{\ell \hat c}-1}{e^{\hat c}-1}+e^{\ell \hat c}g(z),\quad
x'=g(\hat c) \frac{e^{(\ell+1) \hat c}-1}{e^{\hat c}-1}+e^{(\ell+1) \hat c}g(z).
$$
We compute 
$$
\frac{x'}{x}=e^{\hat c} + \frac{g(\hat c)}{x}= e^{\hat c} + O(\frac1x).
$$ 
Thus we have found a sequence $x(\ell)$ for which
$x(\ell+1)/x(\ell)\to e^{\hat c}$, and
$$
\frac{\overline{M}(x(\ell+1))}{\overline{M}(x(\ell))}
=e^{U(x(\ell+1))-U(x(\ell))}\cdot\frac{x(\ell)}{x(\ell+1)}
=e^{\hat c}\frac{x(\ell)}{x(\ell+1)} = 1 + O(\frac1{x(\ell)}) = 1 + O(\hat c^{-\ell}).
$$
Now given an arbitrary $x$ (but large), let $\ell \in \N$ be such that $x=e^{\ell \hat c+z}$, for  
$z\in[0, \hat c]$. Take $x'$ such that $U(x') = U(x) + \hat c$. 
From the above we know that $x'/x = e^{\hat c} + O(1/x)$,
so $x' = e^{\hat c)} x + O(1)$.
Therefore, using the Mean Value Theorem,
\begin{eqnarray*}
 \frac{\overline{M}(e^{\hat c}x)}{\overline{M}(x)} &=& e^{U(e^{\hat c}x) - U(x)} \frac{x}{e^{\hat c}x}
 = e^{U(e^{\hat c}x)-U(x')} \\
 &=& e^{ O(1) U'(\xi) } = e^{ O(e^{-U(\xi)}) } = e^{O(1/\xi)} = 1+O(1/\xi),
\end{eqnarray*}
for some $\xi$ between $e^{\hat c}x $ and $x'$.
Thus $\overline M$ satisfies \eqref{eq-asymp-log-periodic} as required.
\end{proof}

\subsection{Verification of equation \eqref{eq:tail-of-tau-sin-final} and (H3) 
for $\mu_{Y}(\tau >n)$}\label{sec.h3-check}
We now explain how Proposition~\ref{prop-4.3} and its proof
can be used to verify \eqref{eq:tail-of-tau-sin-final}, and in particular assumption (H3). By  Proposition~\ref{prop-4.3} 
equation \eqref{eq-w} holds for $k_n=\lfloor c^n\rfloor$ and we can define 
$\tilde{M}(x)=\lim_{n\to\infty}M_0(k_nx)$, for $x\in(0,\infty)$.

As in (H3), let $A_n=n^{1/\alpha}$, and for $x\in[A_{k_n},A_{k_{n+1}})$
define $\delta(x)=x/A_{k_n}$. To verify \eqref{eq:tail-of-tau-sin-final} and thus (H3), that is
\begin{equation*}
\mu_Y(\tau>x)=x^{-\beta}(\tilde{M}(\delta(x))+h(x)),
\end{equation*}
we need to explain why $\tilde{M}$ satisfies (H3)(ii) and define $h$ so that (H3)(iii) holds. 

From the proof of Proposition \ref{prop-4.3}, we recall that $M_0(x)=\overline{M}(x)+O((\log x)^2/x$ as $x\to\infty$,
where $\overline{M}$ is specified in equation \eqref{eq-integral} (and is a continuous function).
We claim $ M_0(k_n\delta(x))$ is a sequence of uniformly continuous functions, up
to error $O\left(\log(A_{k_n})^2/A_{k_n}\right)$. This latter error can be absorbed into a right continuous function $h(x)$,
with $h(x)=O((\log x)/x^2))$, as desired.


It remains to prove the claim. Consider $x,y\in[A_{k_n},A_{k_{n+1}})$, so that $\delta(x)=x/A_{k_n}$ and
$\delta(y)=y/A_{k_n}$. Recalling that $\overline{M}(x)=e^{U(x)}/x$,
\begin{equation}
\begin{split}
|M_0(\delta(x) k_n)-M_0(\delta(y) k_n)|
&=|\overline{M}(\delta (x) k_n)-\overline{M}(\delta(y) k_n)|+O\left(\log(k_n)^2/k_n\right)\\
&=\left|\frac{e^{U(x k_n)}}{x k_n}-\frac{e^{U(y k_n)}}{y k_n}\right|+O\left(\log(k_n)^2/k_n\right)\\
&\leq \frac{1}{x k_n}|e^{U(x k_n)}-e^{U(y k_n)}|+e^{U(y k_n)}\frac{|x-y|}{k_n xy}
+O\left(\frac{\log(k_n)^2}{k_n}\right),\\
&\leq C_1\frac{A_{k_n}}{x}|\delta(x)-\delta(y)|+C_2|\delta(x)-\delta(y)|
+O\left(\frac{\log(k_n)^2}{k_n}\right).
\end{split}
\end{equation} 
The last line follows from the fact that for $x,y\approx A_{k_n}$, $e^{U(x k_n)}\approx k_n A_{k_n}$,
and $U(x)$ depends continuously on $x$ (and hence $\delta(x)$), with $C_1$, $C_2$ 
independent
of $k_n$, as required.

\section{Proof of Theorem \ref{thm.singmap} 
and verification of (H0)-(H2) for $f_{M_1}$}\label{proof.thm.sing}
In this section we prove Theorem \ref{thm.singmap} and for the map $f_{M_1}$ 
given by equation \eqref{eq.intermittent2} we verify hypotheses (H0)-(H3) 
(as described in Section \ref{sec-frame}). Hypothesis (H3) is already verified from
Section \ref{sec.h3-check}.
Throughout this section 
we work with $f_{M_1}$, and hence drop the subscript $M_1$.

To describe the organisation of proof, we explicitly
construct a partition $\mathcal{Q}$ of $Y$, and a stopping time 
$\varphi:\mathcal{Q}\to\mathbb{N}$, such that for all $\omega\in\mathcal{Q}$,
$f^{\varphi}(\omega)=Y$ bijectively, and with bounded distortion.
Due to the presence of the singularity set $\mathcal{C}$, the 
stopping time $\varphi$ will not (in general) coincide with the 
first return time $\tau$. In the first part of the construction we define 
an auxiliary partition $\mathcal{P}$ and a 
stopping time $T$, such that for all $\omega\in\mathcal{P}$, we have $T(\omega)<\infty$, and $|f^{T(\omega)}(\omega)|\geq\delta$.
The constant $\delta$ will be chosen to reflect that an interval $\omega\subset Y$ reaches a certain large scale under $f^{T}$, and
with bounded distortion. We then estimate $\mathrm{Leb}(T>n)$. This is the key technical step. Once we have constructed 
$\mathcal{P}$, it is then a fairly standard argument to link the asymptotics of $\mathrm{Leb}(T>n)$ to that
of $\mu(\varphi>n)$.

\subsection{A combinatorial construction and tail estimates}\label{sec.comb}
We now define a partition $\mathcal{P}$ of $Y$, and a stopping time $T:\mathcal{P}\to\mathbb{N}$ such that for
each $\omega\in\mathcal{P}$ we have $|f^{T(\omega)}(\omega)|\geq\delta$. The constant $\delta$ will be fixed, but for the moment
we can assume that $\delta>\sqrt{\mathcal{D}}$ (and so $\delta$ is chosen so that $|f^{T}(\omega)|$ reaches a definite large scale).
Moreover we construct $\mathcal{P}$ in such a way that $f^k\mid\omega$ is a diffeomorphism, and satisfies bounded distortion estimates
for all $k\leq T(\omega)$. For the following combinatorial construction
we take a finite time partition $\mathcal{P}_n$
of $Y$ defined inductively. All $\omega\in\mathcal{P}_n$ will have $T(\omega)>n$.
We set $\mathcal{P}_0=Y$.
In the following we suppose that $\omega\in\mathcal{P}_n$, and so $|f^n(\omega)|<\delta$ for $n\geq 1$. 

First, we define a fixed partition $\mathcal{A}$ consisting of intervals
associated to the sequence $x_{n}$ described in Section \ref{sec.first-tail}, 
and for which 
the restriction of $f$ to these intervals in $\mathcal{A}$
is continuous. We set $J_n=[x_{n+1},x_n]$. 
If $s_{\ell}\not\in(x_{n+1},x_n)$ then we put $J_n\in\mathcal{A}$. If $s_{\ell}\in(x_{n+1},x_n)$, then
we write $J_n=J^{+}_n\cup J^{-}_{n}$, with $J^{-}_n=[x_{n+1},s_{\ell}]$ and $J^{+}_n=[s_{\ell},x_{n}]$. We put
$J^{\pm}_n\in\mathcal{A}$. Thus if for all $k\leq n$, $f^k(x)$ lies in the interior of intervals in $\mathcal{A}$,
then $f^n$ is (locally) continuous and differentiable at $x$.
We consider the following cases.
\begin{enumerate}
\item Suppose that $f^{n+1}(\omega)\cap\left(I_{\mathcal{D}}\cup\{1/2\}\right)=\emptyset$, and $|f^{n+1}(\omega)|<\delta$. Then
we put $\omega$ into $\mathcal{P}_{n+1}$. If $|f^{n+1}(\omega)|>\delta$ then $\omega$ has reached large scale, and
we put $\omega$ into $\mathcal{P}$ and set $T(\omega)=n+1$. This component is then taken out of circulation.
\item Suppose that $f^{n+1}(\omega)\cap\left(I_{\mathcal{D}}\cup\{1/2\}\right)\neq\emptyset$. Then we have various cases.
Again if $|f^{n+1}(\omega)|>\delta$ then we put $\omega$ into $\mathcal{P}$ and set $T(\omega)=n+1$. Otherwise,
we subdivide $\omega$ by intersecting the image $f^{n+1}(\omega)$ with $\mathcal{A}$ and then pull back. This is done as
follows:

$\bullet$ Suppose that $f^{n+1}(\omega)$ straddles at least three elements of $\mathcal{A}$. Then we subdivide $\omega$
into pieces $\omega'\subset\omega$ in such a way that $f^{n+1}(\omega')=J_r$, for some $J_r\in\mathcal{A}$.
If $f^{n+1}(\omega)\cap\mathcal{C}\neq\emptyset$, then we ensure each $\omega'$ is contained in the relevant $J^{\pm}_r$
accordingly. We assign a \emph{return depth} value $r$ to each $\omega'$ at time $n+1$. Now by construction of $\mathcal{A}$, the time
for points $x\in f^{n+1}(\omega')$ to escape from $[0,1/2]$ is of the order $r$. (The same is true for the asymptotic time to escape $I_{\mathcal{D}}$, i.e. in the case $r\to\infty$). For a component $\omega'$ such that $f^{n+1}(\omega')\subset J_r$, it may intersect $\mathcal{C}$
as it evolves under further iteration under $f$ up to time $r$. The number $N_r$ of intersections of $f^{k+n+1}(\omega')$ with $\mathcal{C}$ up to time $k\leq r$ is given by the relation:
$$r^{-1/\alpha}\approx e^{-N_r c_1},$$
and so $N_r\approx \frac{1}{c_1\alpha}\log r$. For each intersection of $f^{k+n+1}(\omega')$ with a given $s_{\ell}\in\mathcal{C}$ the
component $\omega'$ is chopped into further pieces (at most two). Hence up to time $n+1+r$, at most
$2^{N_r}\approx r^{\log 2/\alpha c_1}$
pieces are generated per component $\omega'\subset\omega$ that satisfies $f^{n+1}(\omega')=J_r$. Each of these $\leq 2^{N_r}$ pieces
$\omega''\subset\omega'$ are then put into $\mathcal{P}_{n+1}$. They are all declared to have return depth $r$, and 
for each piece $\omega''$, $f^{n+r}\mid\omega''$ is a diffeomorphism. The time a component spends in $I_{\mathcal{D}}$ before escaping
is comparable to $r-\mathcal{D}^{-\alpha}$, and hence asymptotically equal to $r$. A component cannot reach large scale during this time, 
since we assume $\mathcal{D}\ll\delta$.

$\bullet$ Suppose that $J_r\subset f^{n+1}(\omega)\subset J_{r-1}\cup J_r\cup J_{r+1}$. Then we assign a return depth value $r$ to $\omega$, 
and subdivide into pieces $\omega''\subset\omega$ as in the previous item, and so that $f^{n+r}\mid\omega''$ is a diffeomorphism. Each
$\omega''$ is put into $\mathcal{P}_{n+1}$.

$\bullet$ If $f^{n+1}(\omega)\subset J_{r-1}\cup J_r$, the we assign a return depth value $r$, and chop $\omega$ accordingly based upon
intersections with $\mathcal{C}$ as we iterate up to time $n+r$.

$\bullet$ If $f^{n+1}(\omega)\cap\{1/2\}\neq\emptyset$, then we chop $\omega$ into two pieces $\omega^{L}$ and $\omega^R$
with $f^{n+1}(\omega^L)\subset[0,1/2]$ and $f^{n+1}(\omega^R)\subset[1/2,1]$. We put
$\omega^L$ and $\omega^R$ into $\mathcal{P}_{n+1}$, but assign no return depth value. This we call an \emph{inessential} chop.
\end{enumerate}

Thus for each $\omega\in\mathcal{P}_n$, there is an itinerary $(t_1,r_1),\,(t_2,r_2),\ldots,(t_s,r_s)$, where
$t_i$ are the sequence of times corresponding to successive returns to $I_{\mathcal{D}}$, and $r_i$ are the associated return depths.
Due to chopping procedure described above, we in fact generate many components with the same itinerary. In fact the cardinality
of such components is equal to:
\begin{equation}\label{eq.count1}
\prod_{i=1}^{s} 2^{N_{r_i}}\leq C^s \prod_{i=1}^{s} 
r_{i}^{\left(\frac{\log 2}{c_1\alpha}\right)},
\end{equation}
for some uniform constant $C>0$. We should also account for inessential chops. However in the case where 
$f^{n+1}(\omega)\cap\{1/2\}\neq\emptyset$, the left piece $\omega^{L}$ is mapped to an interval that contains the 
hyperbolic fixed point $x=1$. Thus in a finite number of iterations, the image of $\omega^{L}$ eventually covers $[1/2,1]$
(diffeomorphically). The component 
$\omega^{R}$ is mapped into $I_{\mathcal{D}}$ (at time $n+2$) and it's combinatorial structure is then treated via the above algorithm. 
Thus intersections of components with $x=1/2$ do not add a significant contribution to the cardinality $\mathcal{P}_n$ relative to intersections with $I_{\mathcal{D}}$. We can absorb the counting of such components into the constant $C$ given in equation \eqref{eq.count1}. 

\subsection{Tail estimates on the stopping time $T$}
In this section we prove the following result:
\begin{prop}\label{prop.tail-T}
The following tail estimate for $T$ holds:
$$\mathrm{Leb}(\{\omega\in\mathcal{P}:T(\omega)>n\})=O\left(n^{-\frac{1}{\alpha}}\right).$$
\end{prop}
The idea of proof follows that of \cite[Section 3.3]{BLvS}. In our case we have good control
of the derivative along the orbit for a given itinerary. However, the complication is that we have a countable set $\mathcal{C}$
of discontinuities. This means that we generate a large number of components with 
a given itinerary. We show
that expansion wins over chopping. By choosing $c_1$ sufficiently large, we ensure 
(on average) that 
components grow to large scale relative to being chopped up. 

\begin{proof}
To prove Proposition \ref{prop.tail-T}, we define two partitions $\mathcal{P}^{(1)}_n$ and $\mathcal{P}^{(2)}_n$ as follows.
Let $\eta\in(0,1)$ be a small constant to be fixed. For $\omega\in\mathcal{P}_n$ consider its itinerary of return depths
$(r_1,\ldots,r_s)$ defined up to time $n$.  Then we set
\begin{equation}
\mathcal{P}^{(1)}_n=\{\omega\in\mathcal{P}_n: T(\omega)>n,\sum_{i=1}^{s} r_i\leq \eta n\},
\end{equation}
and
\begin{equation}
\mathcal{P}^{(2)}_n=\{\omega\in\mathcal{P}_n: T(\omega)>n,\sum_{i=1}^{s} r_i> \eta n\},
\end{equation} 
Note that $r_1+\ldots +r_{s-1}\leq n$, but it is possible that $r_s\geq n$. We have the following lemma (for reference
compare to \cite[Lemma 3.5]{BLvS}).
\begin{lemma}\label{lem.comb1}
There exists $\theta>0$ such that
\begin{equation}
\sum_{\omega\in \mathcal{P}^{(1)}_n}|\omega|= O\left(e^{-\theta n}\right).
\end{equation}
\end{lemma}
\begin{proof}
Consider $\omega\in\mathcal{P}_n$ with itinerary of return depths
$(r_1,\ldots,r_s)$ defined up to time $n$. Since a return depth value is only associated to $\omega$ on a return to $I_{\mathcal{D}}$,
it follows that $r_i\geq\mathcal{D}^{-\alpha}>1$ for each $i\leq s$. Moreover this puts a bound on $s$, namely that
$s\leq \mathcal{D}^{\alpha}n$. It is also possible for this sequence to be empty if
$f^k(\omega)\cap I_{\mathcal{D}}=\emptyset$ for all $k\leq n$. For iterates evolving outside of $I_{\Delta}$, the map is
uniformly expanding, and $Df^k(x)\geq C\lambda^{k}$ if the orbit of $x$ resides outside of 
$I_{\mathcal{D}}$ up to time $k$. 
However, the rate constant $\lambda$ depends on $\mathcal{D}$. 
A blunt lower bound on the $\lambda$ is given by 
$$\lambda\geq Df_{\alpha}(\mathcal{D})=1+C_1\mathcal{D}^{\alpha},$$
where $C_1$ depends only on $\alpha$ and $c$.
More generally, if $x\in I_{\mathcal{D}}$ is such that $Df^n$ is defined at $x$, 
$f^k(x)\in(0,1/2)$ for all $k\leq n$, and
$f^n(x)\not\in I_{\mathcal{D}}$, then we have
$$Df^n(x)\approx n^{1+\frac{1}{\alpha}}.$$
A more refined analysis of the derivative can be obtained (but we do not need it here). Thus we have the following estimate:
\begin{equation}\label{eq.p1-est1}
\begin{split}
\sum_{\omega\in \mathcal{P}^{(1)}_n}|\omega| &\leq \sum_{s=1}^{n\mathcal{D}^{\alpha}}\sum_{k=0}^{\eta n}
\sum_{\sum r_i=k}|\omega(r_1,\ldots,r_s)|\\
&\leq \sum_{s=1}^{n\mathcal{D}^{\alpha}}\sum_{k=0}^{\eta n} N_{k,s}\lambda^{n-\eta n}
\left(\prod_{i=1}^{s} Cr^{-1-\frac{1}{\alpha}}_i
\cdot 2^{\frac{1}{c_1\alpha}\log r_i}\right)\\
&\leq \sum_{s=1}^{n\mathcal{D}^{\alpha}}\sum_{k=0}^{\eta n} C^sN_{k,s}\lambda^{n-\eta n}, 
\end{split}
\end{equation}
where $C>0$ is a uniform constant.
In the first line if equation \eqref{eq.p1-est1}, we sum over all relevant $\omega(r_1,\ldots,r_s)$ with
$\sum_{i=1}^{s} r_i=k$. In the second line, we let $N_{k,s}$ denote the number of (positive) integer sequences
$r_1,\ldots r_s$ with $\sum r_i=k$. The constant $\lambda$ is the expansion during the iteration outside $I_{\mathcal{D}}$.
The expression inside the term $\prod(\cdot)$ is formed by counting the number of components with itinerary $r_i$, and
calculating the derivative along the orbit in a $J_{r_i}$ (up to time $r_i$). The net contribution 
of each term inside the product is equal to:
\begin{equation}\label{eq.beta-ref}
r_{i}^{-1-\frac{1}{\alpha}(1-c^{-1}_1\log 2)}:=r_{i}^{-\beta_1},
\end{equation}
and we assume $c_1$ is large enough that $\beta_1>2$.
Hence the product is insignificant in the case of $\mathcal{P}^{(1)}_{n}$. We now estimate $N_{k,s}$. 
A standard counting argument implies that
$$N_{k,s}\leq {{k}\choose{s}}.$$
Here $0\leq k\leq \eta n$ and $s\leq\mathcal{D}^{\alpha} n$. However, this forces 
$s\leq \mathcal{D}^{\alpha}k$, since
each $r_i\geq\mathcal{D}^{-\alpha}$. In particular this implies that $N_{k,s}$ is clearly an overestimate, and further optimization
is possible (but the final estimate we obtain is sufficient for our purpose).  An application of Stirling's formula
implies that
\begin{equation*}
\begin{split}
N_{k,s} &\leq \left(1+\frac{2s}{k}\right)^k\cdot\left(\frac{k}{s}\right)^{s} \leq \exp\{2s+s(\log k-\log s)\}.
\end{split}
\end{equation*}

Since $\exp\{2s + s(\log k - \log s)\}$ is increasing in $s$ and $s \leq \epsilon k$ for some
$\epsilon<\mathcal{D}^{\alpha}$, we have:
\begin{equation*}
N_{k,s} \leq \exp\{\epsilon k(2-\log\epsilon)\}.
\end{equation*}
Hence
\begin{equation*}
\begin{split}
\sum_{\omega\in \mathcal{P}^{(1)}_n}|\omega| &\leq 
\sum_{s=1}^{n\mathcal{D}^{\alpha}\eta}\sum_{k=0}^{\eta n} C^s \exp\{\mathcal{D}^{\alpha}k(1-\log\epsilon)\}\cdot \lambda^{n-\eta n}\\
&\leq n^2\mathcal{D}^{\alpha}\eta^2\cdot C^{n\mathcal{D}^{\alpha}\eta}\exp\{\mathcal{D}^{\alpha}\eta n(1-\log\epsilon)\}\cdot \lambda^{n-\eta n}\\
&=O(e^{-\theta n}),
\end{split}
\end{equation*}
for some $\theta>0$. This bound follows from the fact that $\lambda$ is independent of $\eta$. However, note that the lower 
bound on $\lambda$ does depend on $\mathcal{D}$. This completes the proof of Lemma \ref{lem.comb1}.
\end{proof}

We now consider elements in $\mathcal{P}^{(2)}_n$. We have the following lemma (for reference
compare to \cite[Lemma 3.6]{BLvS}).
\begin{lemma}\label{lem.comb2}
The following estimate holds:
\begin{equation}
\sum_{\omega\in \mathcal{P}^{(2)}_n}|\omega| = O\left(n^{-\frac{1}{\alpha}}\right).
\end{equation}
\end{lemma}
\begin{proof}
Again, consider $\omega\in\mathcal{P}^{(2)}_n$ 
with itinerary of return depths
$(r_1,\ldots,r_s)$ defined up to time $n$, and with associated sequence of return times $(t_1,\ldots,t_s)$.
Since $\sum_{i=1}^s r_i\geq \eta n$, a simple argument using the pigeon hole principle implies that there exists $j\leq s$ such 
that $r_j\geq \eta n/(2j^2).$ Consider now a set $\tilde{\omega}$ formed by a union of components in $\mathcal{P}^{(2)}_n$ which share with $\omega$
the same itinerary of return times $(t_1,\ldots,t_j)$, and itinerary of return depths $(r_1,\ldots, r_{j-1})$, but having 
$r_j\geq \eta n/(2j^2).$ Thus $\tilde\omega$ is formed by a concatenation of components 
(that include $\omega$), and $f^{t_j}$ maps
$\tilde\omega$ diffeomorphically into an interval $(u,v)\subset [0,x_{r}]$ with
$x_r<x_{r_j}$. Hence $|f^{t_j}(\tilde\omega)|\leq |[0,x_{r_j}]|.$
Let $\tilde\omega(r_1,\ldots, r_{j-1},r)\subset\tilde{\omega}$ 
be a component in $\mathcal{P}^{(2)}_n$ with $f^{t_j}(\tilde\omega(r_1,\ldots,r_{j-1},r)$ 
having depth $r\geq \eta n/(2j^2)$.
We obtain
\begin{equation}\label{eq.p2-est1}
\begin{split}
\sum_{\omega\in \mathcal{P}^{(2)}_n}|\omega| &\leq \sum_{j=1}^{n}
\sum_{(r_1,\ldots,r_s)}|\tilde\omega(r_1,\ldots,r_j)|\\
 &\leq \sum_{j=1}^{n}\sum_{r\geq n/2j^2}|[0,x_{r}]|
\sum_{(r_1,\ldots,r_{j-1})}\prod_{i=1}^{j-1}r^{-\beta_1}_i\\
&\leq \sum_{j=1}^{n}\sum_{r\geq n/2j^2} Cr^{-1-\frac{1}{\alpha}} 2^{-\beta_1 j} 
\sum_{(r_1,\ldots,r_{j-1})}\prod_{i=1}^{j-1}\left(\frac{2}{r_i}\right)^{\beta_1},
\end{split}
\end{equation}
where we recall $\beta_1>2$ from equation \eqref{eq.beta-ref}, and $C>0$ is a uniform constant.
To manipulate the product term we use the following sublemma:
\begin{sublemma}\label{slem.prod}
Suppose $U_i$ is a sequence with $\sum_i U_i<\infty$. Then for all $\zeta>0$ there exists $n_0$ such that
$$\sum_{s\geq 1}\sum_{\underset{n_i\geq{n}_0}{(n_1,\ldots, n_s)}}\prod_{n_i}\zeta U_i\leq 1.$$
\end{sublemma}
We apply Sublemma \ref{slem.prod} to the sequence $U_i=i^{-\beta_1}$ with $\zeta=2^{\beta_1}$, and choose $\mathcal{D}$ small
enough that $\mathcal{D}^{-\alpha}>n_0$. Thus continuing with the estimate in equation \eqref{eq.p2-est1} we obtain:
\begin{equation}\label{eq.p2-est2}
\sum_{\omega\in \mathcal{P}^{(2)}_n}|\omega|\leq \sum_{j=1}^{n} (\eta n/2j^2)^{-\frac{1}{\alpha}} 2^{-\beta_1 j}
=O\left( n^{-\frac{1}{\alpha}}\right).
\end{equation}

\end{proof}

The proof of Proposition \ref{prop.tail-T} now follows from Lemmas \ref{lem.comb1} 
and \ref{lem.comb2}. We obtain that $|\{T(\omega)>n\}|=O(n^{-1/\alpha}).$
We note that the singularity set $\mathcal{C}$ does not affect this power
law asymptotic. The singularity set affects the $O(\cdot)$ multiplier through
the constant $\beta_1$.
\end{proof}

\subsection{Completing the proof of Theorem \ref{thm.singmap}}
Given the tail estimate of Proposition \ref{prop.tail-T}, the remaining part of the proof follows a standard argument, e.g. 
\cite{BLvS, Young}. These arguments don't tend to rely heavily on the precise form of the map.
In the following lemma, the significance of reaching size $\delta$ is made transparent. In particular we show that
for $\omega\in\mathcal{P}$, a definite fraction of $\omega$ diffeomorphically maps onto $Y$ in finite time beyond $T(\omega)$.
We state the following.
\begin{lemma}\label{lem.largescale}
There exists $t_0>0$, $\xi>0$, such that for all $\omega\in\mathcal{P}$, there exists $\tilde{\omega}\subset\omega$
with the following properties:
\begin{enumerate}
\item $f^{T(\omega)+t}(\tilde\omega)=Y$, for some $t\leq t_0$, and this action is a diffeomorphism with bounded distortion.
\item $|\tilde\omega|\geq \xi|\omega|$;
\item Both components of $f^{T(\omega)}(\omega\setminus\tilde\omega)$ are of size $\delta/3$.
\end{enumerate}
\end{lemma}
\begin{proof}
The proof of the lemma follows from the fact that $\mathcal{D}\ll\delta$, e.g. $\delta>\sqrt{\Delta}$ will suffice. 
The cases to consider include i) when $f^{T(\omega)}(\omega)\cap\{1/2\}\neq\emptyset$, 
ii) when $f^{T(\omega)}(\omega)\cap I_{\mathcal{D}}\neq\emptyset$, or iii) when
$f^{T(\omega)}(\omega)\cap\left(\{1/2\}\cup I_{\mathcal{D}}\right)=\emptyset.$ 
For example in the case $f^{T(\omega)}(\omega)\cap I_{\mathcal{D}}\neq\emptyset$ we consider the intersection of
$f^{T(\omega)}(\omega)$ with intervals of the form $[x_{n},x_{n+1}]$ with $x_{n}>\mathcal{D}$. For $n$ large we have 
$x_{n}\approx n^{-1/\alpha}$. Given $\xi>0$, we can choose $n$ so that $n^{-1/\alpha-1}\approx\xi$, and moreover that
both components of $f^{T(\omega)}(\omega)\setminus[x_{n+1},x_n]$ are larger than $\delta/3$. The other cases reduce to similar scenarios.

For the distortion, we note that the Schwarzian derivative $Sf = f''' f' - \frac32 (f'')^2 = -6$ for $x > \frac12$ and 
$=C_0(\tilde\alpha^2-1)\tilde\alpha x^{\tilde\alpha-2} - 
C^2(\tilde\alpha+1)^2\tilde\alpha \frac{2+5\tilde\alpha}{2} x^{2(\tilde \alpha-1)} < 0$
for $x < \frac12$ inside each interval where$f$ is continuous.
Here $\tilde \alpha = \alpha-c^{-1}_1\log 2 \in (0,1)$.
Therefore $f^{T(\omega)}(\omega)$ has negative Schwarzian derivative for each $\omega$ in the statement of this lemma.
Since both components of $f^{T(\omega)}(\omega\setminus\tilde\omega)$ are of size $\delta/3$,
the Koebe Principle (see \cite{MeSt}) gives boundedness of distortion of $f^{T(\omega)}$ on $\tilde\omega$, and the remaining $\leq t_0$ will not destroy this.
\end{proof}

\begin{rmk} 
\label{rmk-GMfM2}An argument similar to the Proof of \ref{lem.largescale} shows that  the Schwarzian derivative of  $f_{M_1}$, $f_{M_1}^1$ is negative, yielding the 
required distortion properties. Since the big image property is automatically satisfies in the setup of the first induced maps $F$ with $F=(f_{M_1})^\tau$ and  $F=(f_{M_1}^1)^\tau$,
in these cases the map $F$ is Gibbs Markov.
\end{rmk}

\subsubsection*{Tail of the return time function $\varphi$}
Thus, in Lemma \ref{lem.largescale} the left and right components $\omega^{L}$, $\omega^R$ of $\omega\setminus\tilde\omega$ are treated as new starting intervals. We then apply the combinatorial algorithm to each of the two components separately, i.e.
we put $\mathcal{P}_0=\omega^{L,R}$, and then find $|\{x:T(x)>n|x\in\omega^{L,R}\}$. For the central component $\tilde\omega$, this now becomes an element of $\mathcal{Q}$ for 
the Gibbs-Markov map $F:=f^\varphi$, and we define $\varphi(\tilde\omega)=T(\tilde\omega)+t(\omega)$.
This process is repeated, and thus each element of $\mathcal{Q}$ has an associated itinerary of large scale times
$T_1(\omega),\ldots T_n(\omega),\ldots$, before return time $\varphi(\omega)$.
For $x\in Y$, we write $T_1(x):=T(x)$, and define $\{T_{n}\}$ recursively via $T_{n+1}=T_{n}(x)+T(f^{T_n(x)}(x))$. 
Thus for all $\omega\in\mathcal{Q}$, there exists $s\geq 1$, and $t\leq t_0$ such that we have $\varphi(\omega)=T_s(\omega)+t(\omega)$.
and gives rise to a further sequence of partitions $\mathcal{Q}_n$ and stopping times $T_1,\ldots, T_n,\ldots$ We have the following
lemma:
\begin{lemma} For every $i$,
$$\mathrm{Leb}(\{x:T_{i+1}(x)\,\textrm{exists and}\,T_{i+1}>T_i+k\mid T_i\})\leq D_{\delta}\mathrm{Leb}(T>k),$$
where $D_{\delta}>0$ is a constant. 
\end{lemma}
\begin{proof}
This is a standard argument using bounded distortion, the definition of $T$ and Lemma \ref{lem.largescale}.
See \cite[Lemma 4.4.]{BLvS}, and also \cite{Young}.
\end{proof}

Now recall each $\omega\in\mathcal{Q}$ has the sequence:
$$0=T_0<T_1<\ldots <T_{s(\omega)}<\varphi(\omega).$$
Let $\mathcal{Q}^{(n)}=\{\omega\in\mathcal{Q}:\varphi(\omega)>n\}$, and let
$$\mathcal{Q}^{(n)}_i=\{\omega\in\mathcal{Q}^{(n)}:T_{i-1}<n\leq T_i\}$$
denote the set of elements in $\mathcal{Q}^{(n)}$ which have exactly $i-1$ large scale times before $n$.
We let
$$|\mathcal{Q}^{(n)}_i|=\sum_{\omega\in\mathcal{Q}^{(n)}_i}|\omega|,\quad |\mathcal{Q}^{(n)}|=\sum_{i=1}^{n}|\mathcal{Q}^{(n)}_i|.$$

\subsubsection*{Completing the proof of Theorem \ref{thm.singmap}.}
\label{subsec-comlpthdyn}
Let $\eta>0$ be a small constant, and write:
$$|\{\varphi>n\}|=|\mathcal{Q}^{(n)}|=\sum_{i=1}^{\eta n}|\mathcal{Q}^{(n)}_i|+\sum_{i=\eta n}^{n}|\mathcal{Q}^{(n)}_i|.$$
By Lemma \ref{lem.largescale} we have:
$$|\mathcal{Q}^{(n)}_i|/|\mathcal{Q}^{(n)}_{i-1}|\leq 1-\xi,$$
and therefore
$$\sum_{i=\eta n}^{n}|\mathcal{Q}^{(n)}_i|\leq \frac{1}{\xi}(1-\xi)^{\eta n}.$$
To estimate $\sum_{i=1}^{\eta n}|\mathcal{Q}^{(n)}_i|$ we again use a combinatorial argument (and see \cite{BLvS,Young} for similar).
For each element of $\mathcal{Q}^{(n)}_i$, we can assign an itinerary $(k_1,\ldots,k_i)$ with
$\sum_j k_j=n$, and $k_j=T_{j}-T_{j-1}$ for all $j\leq i-1$. We put $k_i=n-T_{i-1}$. A pigeon hole argument implies that
for any such itinerary, there exists $k_j$ with $k_j\geq n/i$. Let
$$\mathcal{Q}^{(n)}_{i,j}=\{\omega\in\mathcal{Q}^{(n)}_i, k_{j'}<n/i,\,\textrm{for $j'<j$ and $k_j\geq n/i$}\}.$$
We have the following
\begin{equation*}
\begin{split}
\sum_{i=1}^{\eta n}|\mathcal{Q}^{(n)}_i| & =\sum_{i=1}^{\eta n}\sum_{j=1}^{i}|\mathcal{Q}^{(n)}_{i,j}|
\leq \sum_{i=1}^{\eta n} i(1-\xi)^{i-1}|\{T>n/i\}\\
&\leq \sum_{i=1}^{\eta n} i(1-\xi)^{i-1}\cdot \left(\frac{i}{n}\right)^{\frac{1}{\alpha}}
=O(n^{-\frac{1}{\alpha}}),
\end{split}
\end{equation*}
which gives the required bound. This completes the tail estimate for the return time function, 
and hence the proof of Theorem \ref{thm.singmap}.

\subsection{Verification of (H0)-(H2) for $f_{M_1}$.}
\label{subsec-ver}
We now verify (H0)-(H3) for the map $f_{M_1}(x)$.
Firstly, (H0) immediately follows from the statement of Theorem 
\ref{thm.singmap}.
To verify (H1), we can adapt \cite[Lemma 3.6]{BLvS} via 
Lemma \ref{lem.comb2} to study the asymptotics of the re-inducing time
$\rho:Y\to\mathbb{N}$, (see \cite{BT18}). We obtain
$$\mu_0(\rho > L : \varphi > n)=O\left(\lambda^{L}_0 n^{-\alpha}\right),$$ 
for some 
$\lambda_0 \in (0,1)$. This then leads to
$$
\int_{\varphi > n} \rho \ d\mu_0\leq C_1 \sum_{L \geq 0} \lambda^{L}_{0} n^{-\alpha}
=O(n^{-\alpha})\leq C_2 \mu_0(\varphi > n),
$$
for uniform constants $C_1, C_2$.
To verify (H2), we note that the combinatorial algorithm
described in Section \ref{sec.comb} requires that for $\omega\in\mathcal{P}$
we either have $f^{k}(\omega)\subset Y$ of $f^{k}(\omega)\cap Y=\emptyset$.

\section{Proof of Theorem~\ref{thm-infst} for larger classes of sequences}
\label{sec-exactst}

Recall  $f_{M}$  with $M=M_1$ or $M=M_2$ defined in \eqref{eq.intermittent2} and~\eqref{eq:def-M-sine}. 
Let $k_n=\lfloor c^n\rfloor$ with $c=e^{\alpha c_1}$ and  $c=e^{\alpha c_2}$, respectively.
As shown in Subsection~\ref{sec.h3-check}, (H3) holds along $k_n$. We recall that  in the setup of $f_{M_2}$, the stronger assumption (H4)  holds (by an argument  similar to the one  Subsection~\ref{sec.h3-check} with an easier analysis since $\varphi\equiv\tau$).

Regarding the verification of (H5) on the tail of the induced observable we recall the following result.
Recall that  $g= g_X \circ \pi$ and that $g_{Y}=\sum_{j=0}^{\varphi-1} g\circ f_\Delta^j$.  
\begin{lemma}
\label{lemma-obs}Let $x_0\in X$ be the indifferent periodic point of $f_M$. Let $g_X:X\to\R$ be  a H{\"o}lder observable of exponent $\nu>1/\beta-1/2$
and suppose $g_X(x_0)\ne 0$.
Then $\mu_Y(|g_Y|>t)=C\mu_Y(\varphi>t)(1+o(1))$ as $t\to\infty$ for some $C>0$.
\end{lemma}
\begin{proof}
The argument goes as in~\cite[Proof of Theorem 1.3]{Gouezel04b}. 
We recall the main elements for completeness since the argument there is written for first return times.

Set $\tilde g_Y=\sum_{j=0}^{\varphi-1} g_X\circ f^j$. 
Let $y,k$ so that $\varphi(y)=k$, assume $g_X(x_0)=C>0$ and write
\[
\tilde g_Y(y)= Ck+ \sum_{j=0}^{k-1} \left(g_X ( f^j(y))- g_X(f^j(x_0))\right).
\]
Since $g_X$ is $\nu$-H{\"o}lder,
$\sum_{j=0}^{k-1} |g_X ( f^j(y))- g_X(f^j(x_0))|\le |g_X|_\nu \sum_{j=0}^{k-1} 
\left|f^j(y)-f^j(x_0)\right|^\nu$.
Since $x_0$ is the indifferent  periodic point, 
$\sum_{j=1}^{k}\left|f^j(y)-f^j(x_0)\right|^\nu\ll \sum_{j=1}^{k} (k-j)^{-\beta\nu}\ll k^{1-\nu\beta}$ for all $1\le j\le k$. 
Hence, $\tilde g_Y(y)= C\varphi(y)+O(\varphi(y)^{1-\nu\beta})$. 

Since $\nu>1/\beta-1/2$,  $O(\varphi(y)^{-\nu\beta})=O(\varphi(y)^{-\beta(1/\beta-1/2)})=O(\varphi(y)^{-(1-\beta/2)})= o(1)$, where in the last equality we have used that
and $\beta<2$. Thus, $\mu_Y(|\tilde g_Y|>t)=C' \mu_Y(\varphi>t)(1+o(1))$, as $t\to\infty$. 
This gives the statement on $\tilde g_Y$ and, thus, on $g_Y$ in the case $C>0$.
The case $C<0$ goes similarly.~\end{proof}

Given the already introduce terminology, here we obtain Theorem~\ref{thm-infst} ii) for larger classes of sequences. More precisely, similar to~\eqref{eq:def-delta},  for $x > 0$ (large)  define 
\begin{equation}
\label{eq:gamma2}
\gamma_x=\frac{x}{k_n}, \quad \text{ where } k_{n-1} <  x \leq k_n.
\end{equation}
As explained below, given  $\gamma_x$ as in \eqref{eq:gamma2}, whenever $\gamma_{q_r} \stackrel{cir}{\to} \lambda\in (c^{-1}, 1]$,

\begin{align}
\label{eq-larger}
\frac{\sum_{j=0}^{q_r-1} v\circ f_{M_2}^j}{(q_r)^{1/\alpha}}\overset{d}{\longrightarrow} V(\alpha, \lambda),\mbox{ as } r\to\infty,
\end{align} where $V_\lambda$ is as in~\eqref{eq-conv-subseq}.

\begin{pfof}{Theorem~\ref{thm-infst}} For $f_{M_1}$, the abstract assumptions (H0)--(H3) are verified in Subsections~\ref{subsec-ver} and~\ref{sec.h3-check}.
Thus, Theorem~\ref{prop-main} applies to $f_{M_1}, f_{M_2}$ with  norming sequence $(k_n)^{1/\alpha}$ giving also the additional  result~\eqref{eq-larger}.
We note that the statement on the non trivial distribution
limit for $f_{M_1}$ is along some sequence $m_r$ and no identification of $m_r$ is claimed.
 \end{pfof}

\section*{Acknowledgement}
D. Coates was supported by an EPSRC studentship. M. Holland 
was partially supported from the EPSRC, no. EP/P034489/1. In the later stages 
of preparation of this paper D. Terhesiu 
was partially supported by the EPSRC, no. EP/S019286/1.
We thank the referee for useful remarks that helped us improve several explanations.

\bibliographystyle{abbrv}
\bibliography{wobbly}

\end{document}